\documentclass[12pt,reqno]{amsart}
\usepackage{geometry,amsmath,amssymb,amsthm,amsfonts,enumerate,hyperref,comment}
\usepackage{verbatim}
\usepackage{tensor}

\newgeometry{vmargin={40mm},hmargin={30mm,30mm}}

\theoremstyle{plain}
\newtheorem{theorem}{Theorem}[]
\newtheorem{theorem2}{Theorem}[section]

\newtheorem{lemma}[theorem2]{Lemma}

\theoremstyle{definition}
\newtheorem{definition}[theorem2]{Definition}

\theoremstyle{remark}

\numberwithin{equation}{section}

\newcommand{\tr}{\mathrm{tr}}

\renewcommand{\div}{\mathrm{div}}

\begin{document}
\title{The Fubini--Study metric on an `odd' Grassmannian is rigid}
\author{Stuart James Hall}

\begin{abstract}
	Following the ideas of Gasqui and Goldschmidt, we give an explicit description of the infinitesimal Einstein deformations admitted by the Fubini--Study metric on complex Grassmannians $G_{m}(\mathbb{C}^{n+m})$ with $m,n\geq 2$. The deformations were first shown to exist by Koiso in the 1980s but it has remained an open question as to whether they can be integrated to give genuine deformations of the Fubini--Study metric. We show that when $n+m$ is odd, the answer is no. 
	
\end{abstract}
\maketitle
\section{Introduction}
\subsection{The main theorem and method of proof}
Given a fixed Einstein metric $g_{0}$ on a closed manifold $\mathcal{M}$, a natural question to ask is whether there are other non-isometric Einstein metrics that are in some sense near to $g_{0}$, in particular, whether $g_{0}$ can be deformed through a one-parameter family of geometrically distinct Einstein metrics. The foundational work on this problem was carried out in the 1980s by Koiso in \cite{KoiOsaka1},  \cite{KoiOsaka2}, and \cite{Koi_EMCS}, where such questions were considered for the canonical metrics on compact irreducible symmetric spaces. Koiso was able to demonstrate that, for $m,n\geq 2$, the canonical metric on the complex Grassmannians $G_{m}(\mathbb{C}^{n+m})$ (which we will refer to as the Fubini--Study metric) admits deformations to first order, known as infinitesimal Einstein deformations (EIDs). However, it was left as an open question as to whether any of the EIDs could be integrated to produce a genuine curve of Einstein metrics passing through the Fubini--Study metric; it is this problem that we take up in this article and we are able to demonstate that for `half' of the Grassmannians, the answer is no.
\begin{theorem}\label{Thm:Main}
Let  $m,n\in \mathbb {N}$ with $m,n\geq 2$ and $n+m$ odd. The Fubini--Study metric on  $G_{m}(\mathbb{C}^{n+m})$ is rigid.
\end{theorem}
In order to prove this result, we first describe the deformations of the Fubini--Study metric in a concrete manner. Koiso's original proof for the existence of these EIDs was non-constructive and based upon the analysis of eigenspaces of the Lichnerowicz Laplacian for symmetric spaces using tools coming from representation theory \cite{KoiOsaka1}. In their monograph \cite{GGbook}, Gasqui and Goldschmidt described a concrete construction of the EIDs, focussing mainly upon the case when $n=m$ (though the idea for the construction in other cases is certainly implicit in their work). Their method of proof was also representation theoretic in nature. In Section \ref{Sec:3} we follow the idea of the construction in  \cite{GGbook} but use fairly elementary complex differential geometry to prove the construction yields EIDs of the Fubini--Study metric (see Theorem \ref{Thm:deform} in Section \ref{Sec:3}).\\
\\
In \cite{KoiOsaka2}, Koiso developed the second order obstruction theory for EIDs; in particular, he showed that a certain integral quantity must vanish in order for the deformation to be unobstructed at second order (see Equation \ref{Koiobs} in Section \ref{Sec:2}). Using the explicit complex differential-geometric description of the deformations, we are able to adapt a method of Adler and van Moerbeke \cite{AvM} and compute the obstruction integral for a particular deformation. We can then appeal to the method of \cite{BHMW} and argue that as the integral does not vanish for the particular defomation, when $n+m$ is odd, it does not vanish for any deformation. Thus we conclude all that all the infinitesimal deformations are obstructed and so the metric is rigid.\\
\\
Our proof of Theorem (\ref{Thm:Main}) can be adapted to show that, apart from in the case when $n=m$, most of the EIDs on $Gr_{m}(\mathbb{C}^{n+m})$ when $n+m$ is even are obstructed at second order. It seems likely that the remaining EIDs are also obstructed but at higher orders. However, to the best of the author's knowledge, the formulae for third order variations of the Einstein equations, that would be needed to produce an obstruction to integrability at this order, have not been calculated. 
\subsection{The recent work of Nagy, Semmelmann, and Schwahn}

Theorem \ref{Thm:Main} has recently been proved independently by Schwahn and Semmelmann \cite{SchSem24}. Their method builds on work by Nagy and Semmelmann \cite{NagSem23} who gave a proof for the special case when $m=2$ (here the Grassmannian is also a Quaternionic K\"ahler manifold and this extra structure was exploited). In \cite{NagSem23}, the authors reformulate Koiso's obstruction and express it in a coordinate-free manner. They also study in detail the case of (non-K\"ahler) deformations of K\"ahler-Einstein manifolds and give a new criterion for integrability. It is this new criterion that is used in \cite{SchSem24} and the calculations are accomplished using Lie-theoretic methods which are totally different from the coordinate based calculations we make in this paper. The authors actually accomplish a bit more in \cite{SchSem24} as they characterise the unobstructed infinitesimal deformations in the case that $n+m$ is even as well as deal with the more awkward case when $n=m$.
\subsection{Revisiting Koiso's original example of $\mathbb{CP}^{2n}\times\mathbb{CP}^{1}$}
In the paper \cite{KoiOsaka2}, Koiso was able to demonstrate that the product of Fubini-Study metrics on $\mathbb{CP}^{2n}\times\mathbb{CP}^{1}$ admits EIDs but that none of the deformations is integrable. Koiso's method of proof of this fact is a blueprint for the proof of our Theorem \ref{Thm:Main}: describe the space of deformations explicitly and then use the description to calculate Koiso's obstruction integral. As part of his computations, Koiso used several identities that were proved by using ingenious juggling of standard tricks in Riemannian geometry (e.g. repeated integration by parts then switching appropriate derivatives and manipulating any curvature terms that are introduced). In Section \ref{Sec:5}, we recover these identities using the methods of this paper thus providing a useful check that the somewhat complicated `local' calculations we are making do indeed recover identities proved using a different method.\\
\\
{\it Acknowledgements:} The author began work on this problem during a research stay with Ruadha\'i Dervan at the University of Glasgow in November 2022. Some of the final calculations were completed during a stay in February 2024 at the Isaac Newton Institute as a participant in the programme `New Equivariant Methods in Algebraic and Differential Geometry'. The author would like to thank these institutions for providing excellent research environments and to thank Ruadha\'i for the invitations, his support, and useful discussions about the project. The author would also like to thank Paul Schwahn and Uwe Semmelmann for useful discussions and for the invitation to speak about this project at a workshop on Einstein manifolds organised by them in Stuttgart in October 2023.

\section{Background on Hermitian variations of K\"ahler--Einstein metrics} \label{Sec:2}

The material in this section is well explained in chapter 12 of \cite{Bes}.  
\subsection{First order variations}
Given an Einstein manifold $(\mathcal{M}^{n},g)$, if there were a smooth one-parameter family of Einstein metrics $g(t)$ with $g(0)=g$, then we consider what equations the derivative of the curve $\dot{g}(0)$ would have to satisfy. Of course we can always produce such curves by homothetic scaling or by acting by a one-parameter subgroup of diffeomorphisms so we further require that our curve is not generated by these procedures. Koiso showed that if no such curve exists, then the metric $g$ is isolated in the moduli space of Einstein metrics and we call such a metric {\it rigid}.
\begin{definition}[Infinitesimal Einstein Deformation]
	Let $(\mathcal{M}^{n},g)$ be an Einstein manifold such that ${\mathrm{Ric}(g)=\lambda g}$ with $\lambda>0$. An \textit{infinitesimal Einstein deformation} (EID) is a section $h\in \Gamma(s^{2}(T^{\ast}\mathcal{M}))$ satisfying the following conditions:
	\begin{eqnarray}
	\mathrm{tr}_{g}(h) = 0, \label{trace_free}\\
	\div(h)=0, \label{div_free}\\
	 \Delta h +2\mathrm{Rm}(h) = 0, \label{Lin_Ein}
	\end{eqnarray}
	where $\Delta$ is the connection Laplacian and $\mathrm{Rm}$ is the curvature operator acting on symmetric 2-tensors. We denote the space of EIDs by $\varepsilon(g)$. 
\end{definition}
Equations (\ref{trace_free}) and (\ref{div_free}) ensure that the deformations do not simply come from homothetic scaling or from acting by a one-parameter subgroup of diffeomorphisms; tensors that are both trace and divergence-free are known as transverse trace-free or `TT' tensors. Equation (\ref{Lin_Ein}) is the linearised Einstein equation restricted to the set of TT tensors. As Equation (\ref{Lin_Ein}) is elliptic, the space $\varepsilon(g)$ is finite-dimensional. If $\varepsilon(g) = \{0\}$, then the metric $g$ is rigid.\\
\\ 
If the manifold $(\mathcal{M}^{n},g)$ is K\"ahler--Einstein and the EID $h$ is also invariant with respect to the complex structure $J$, then the equations definining an EID can be conveniently expressed as conditions involving the two-form associated to $h$, $\sigma \in \Omega^{(1,1)}(\mathcal{M})$:
\begin{eqnarray}
\Lambda(\sigma)=0, \\
\bar{\partial}^{\ast}\sigma=0, \\
\Delta_{\bar{\partial}}\sigma = \lambda \sigma,     
\end{eqnarray}
where $\Lambda$ is the adjoint of the Lefschetz operator, $\bar{\partial}^{\ast}$ is the usual adjoint of the Dolbeault operator $\bar{\partial}$, and ${\Delta_{\bar{\partial}}=\bar{\partial}\bar{\partial}^{\ast}+\bar{\partial}^{\ast}\bar{\partial}}$. The translation of the `TT' equations follows from the fact that the complex structure is parallel; the translation of the linearised Einstein condition follows from a Weitzenb\"ock identity for two-forms.\\
\\
We state here the following result of Koiso listing which compact irreducible symmetric spaces admit EIDs.
\begin{theorem2}[Koiso, Theorem 1.1 in \cite{KoiOsaka1}  - see also \cite{GGbook} Proposition 2.40]\label{Infdef_thm}
	Let $(\mathcal{M},g)$ be a compact irreducible symmetric space. Then the space of EIDs $\varepsilon(g)=\{0\}$, except in the following cases:
	\begin{enumerate}[(i)]
		\item $\mathcal{M}=SU_n$ with $n\geq 3$, here $\varepsilon(g)\cong \mathfrak{su}_{n}\oplus\mathfrak{su}_n$,
		\item $\mathcal{M}=SU_n/SO_n$ with $n\geq 3$, here $\varepsilon(g)\cong \mathfrak{su}_n$,
		\item $\mathcal{M}=SU_{2n}/Sp_n$ with $n\geq 3$ here  $\varepsilon(g)\cong \mathfrak{su}_{2n}$,
		\item $\mathcal{M}=SU_{n+m}/S(U_m \times U_n)$ with $n\geq m\geq2$, here $\varepsilon(g)\cong \mathfrak{su}_{n+m}$,
		\item $\mathcal{M}=E_{6}/F_{4}$, here $\varepsilon(g)\cong \mathfrak{e}_{6}$.
	\end{enumerate} 
\end{theorem2}
The complex Grassmannians are spaces (iv) on this list and are the only spaces on the list where the metric is K\"ahler--Einstein (i.e. they are the only Hermitian symmetric spaces on the list).

\subsection{Second order variations}
We denote by $\mathrm{Met}(\mathcal{M})\subset \Gamma(s^{2}(T^{\ast}\mathcal{M}))$ the set of Riemannian metrics on $\mathcal{M}^{n}$. One can consider Einstein metrics as zeros of the Einstein operator ${\mathcal{E}:\mathrm{Met}(\mathcal{M})\rightarrow \Gamma(s^{2}(T^{\ast}\mathcal{M}))}$ given by
 \[
 \mathcal{E}(g) = \mathrm{Ric}(g) - \frac{\int_{\mathcal{M}}\mathrm{Scal}_{g} dV_{g}}{n\mathrm{Vol}(\mathcal{M})}g.
 \]
 If $h$ is an EID, then the curve $g_{1}(t)=g+th$ solves the Einstein equations to first order in the sense that $\mathcal{E}(g_{1}(0))=0$ and
 \[
 \frac{d}{dt}\mathcal{E}(g_{1}(t))\bigg|_{t=0}=0.
 \]
 Koiso investigated the conditions under which it is possible to find $h_{2}\in \Gamma(s^{2}(T^{\ast}\mathcal{M}))$ such that the curve
 \[
 g_{2}(t) =g+th+\frac{t^{2}}{2}h_{2}
 \]
 solves the Einstein equations to second order in the sense that $\dfrac{d^{k}}{dt^{k}}\mathcal{E}(g_{2}(t))\bigg|_{t=0}=0$ for $k=0,1,2$. If it is possible to find such an $h_{2}$ then we say that the EID $h$ is \textit{integrable to second order}.  If it is not possible to find such an $h_{2}$, then we say $h$ is \textit{obstructed at second order}; in particular, $h$ cannot be tangent to a genuine non-trivial deformation of $g$ through Einstein metrics. If all of the EIDs in $\varepsilon(g)$ are obstructed to second order, then $g$ is rigid.
 \begin{lemma}[Koiso, Lemma 4.7 in \cite{KoiOsaka2}]\label{Koilem1}
   	Let $(\mathcal{M},g)$ be an Einstein manifold. Then an EID $h \in \varepsilon(g)$ is integrable to second order if and only if $\mathcal{E}''(h,h) \in \varepsilon(g)^{\perp}$. Here the orthogonal complement is with respect to the $L^{2}$-inner product on $s^{2}(T^{\ast}\mathcal{M})$ induced by $g$.
\end{lemma}
   \noindent Using this result, we see that a necessary condition for $h$ to be integrable is the vanishing of the quantity $\langle \mathcal{E}''(h,h),h\rangle_{L^{2}}$. This quantity was also computed by Koiso.
   \begin{lemma}[Koiso, Lemma 4.3 in \cite{KoiOsaka2}]
   	Let $(\mathcal{M},g)$ be an Einstein metric with Einstein constant $\lambda>0$ and let $h \in \varepsilon(g)$. Then an obstruction to the integrability of $h$ to order two is given by the nonvanishing of the quantity
   	\begin{small}
   		\begin{equation}\label{Koiobs}
   		\mathcal{I}(h) :=2\lambda\langle h_{i}^{k}h_{kj},h_{ij}\rangle_{L^{2}}+3\langle\nabla_{i}\nabla_{j}h_{kl},h_{ij}h_{kl} \rangle_{L^{2}}-6\langle\nabla_{i}\nabla_{l}h_{kj},h_{ij}h_{kl}\rangle_{L^{2}},  
   		\end{equation}
   	\end{small}
   	where each of the brackets denotes the $L^{2}$-inner product induced by the metric $g$ on the appropriate bundle.
   \end{lemma}

\subsection{Koiso's obstruction in complex coordinates}
Pointwise, the terms in the obstruction (\ref{Koiobs}) are of the form
\[
\langle h_{i}^{k}h_{kj}, h_{ij}\rangle = h_{i}^{k}h_{kj}h^{ij} = g^{ip}g^{qj}h_{i}^{k}h_{kj}h_{pq} = h_{i}^{k}h_{k}^{q}h_{q}^{i} = \tr(H^{3}),
\]
\[
\langle \nabla_{i}\nabla_{j}h_{kl},h_{ij}h_{kl}\rangle =  (\nabla_{i}\nabla_{j}h_{kl})h^{ij}h^{kl} = (\nabla_{i}\nabla_{j}h_{kl})h_{pq}h_{rs}g^{ip}g^{jq}g^{kr}g^{ls},
\]
\[
\langle \nabla_{i}\nabla_{l}h_{kj},h_{ij}h_{kl}\rangle =(\nabla_{i}\nabla_{l}h_{kj})h^{ij}h^{kl} =(\nabla_{i}\nabla_{l}h_{kj})h_{pq}h_{rs}g^{ip}g^{jq}g^{kr}g^{ls}.
\]
where $H$ is the symmetric endomorphism associated to $h$ given by ${H^{i}_{j} = g^{ik}h_{kj}}$. It will be convenient to compute these quantities in complex coordinates. The tangent space at a point is an inner product space $(V^{2n},g)$ with an almost complex structure $J$ such that $g(J\cdot,J\cdot)=g(\cdot,\cdot)$. We complexify $V$ and extend the tensors $g$ and $h$ $\mathbb{C}$-linearly in both arguments to obtain a $2n$-complex-dimensional space $V_{\mathbb{C}}$ and tensors $g_{\mathbb{C}}$ and $h_{\mathbb{C}}$. The space $V_{\mathbb{C}}$ splits into the $\pm \sqrt{-1}$-eigenspaces for $J$ and we write ${V_{\mathbb{C}}=V_{\mathbb{C}}^{(1,0)}\oplus V_{\mathbb{C}}^{(0,1)}}$.\\
\\
Given a $g$-orthonormal basis of $V$ of the form $v_{1},Jv_{1},v_{2},Jv_{2},\ldots v_{n},Jv_{n}$, we can form the basis $\{e_{i}\}_{i=1}^{n}$ of $V^{(1,0)}_{\mathbb{C}}$ where   
\[
e_{i}=\frac{1}{2}\left(v_{i}-\sqrt{-1}Jv_{i} \right).
\]
This is  an orthogonal basis of $(V^{(1,0)}_{\mathbb{C}},g_{\mathbb{C}})$ with $\|e_{i}\|=1/2$. The set of conjugates 
\[
\bar{e}_{i}=\frac{1}{2}\left(v_{i}+\sqrt{-1}Jv_{i} \right),
\]
form a basis of $V^{(0,1)}_{\mathbb{C}}$. As the tensors $g$ and $h$ are $J$-invariant, the only non-vanishing terms of the extensions  $g_{\mathbb{C}}$ and $h_{\mathbb{C}}$ are those of the form 
\[
g_{k\bar{l}}:=g_{\mathbb{C}}(e_{k},\bar{e}_{l}) \qquad \mathrm{and} \qquad h_{k\bar{l}} = h_{\mathbb{C}}(e_{k},\bar{e}_{l}). 
\] 
We consider Koiso's quantities but in complex coordinates e.g.
\[
\langle h_{k}^{p}h_{p\bar{l}},h_{k\bar{l}} \rangle  = g^{k\bar{q}}g^{r\bar{l}}h_{k}^{p}h_{p\bar{l}}h_{r\bar{q}} = H_{k}^{p}H_{p}^{r}H_{r}^{k} = \tr(H^{3}).
\]
As the metric is K\"ahler, the Chern connection is the same as the $\mathbb{C}$-linear extension of the Levi-Civita connection. We also note that 
\[
(\nabla_{\cdot}\nabla_{\cdot}h)(JX,JY) = (\nabla_{\cdot}\nabla_{\cdot}h)(X,Y).
\]
  
\begin{lemma}\label{Lem:ComptoRiem}
Let $h\in s^{2}(V^{\ast})$ be $J$-invariant and let $T\in (V^{\ast})^{\otimes 4}$ satisfy
\[
T(-,-,X,Y) = T(-,-,Y,X) \qquad \mathrm{and} \qquad T(-,-,JX,JY)=T(-,-,X,Y),
\]
for all $X,Y \in V$. Then, with the notation defined previously,
\begin{eqnarray}
\langle h_{kp}h^{p}_{l},h_{kl} \rangle = 2\langle h_{k\bar{p}}h^{\bar{p}}_{\bar{l}},h_{k\bar{l}} \rangle, \label{RiemCom_1}\\
\langle T_{klrs},h_{kl}h_{rs} \rangle = 4 \mathrm{Re}(\langle T_{k\bar{l}r\bar{s}},h_{k\bar{l}}h_{r\bar{s}} \rangle),
\label{RiemCom_2}\\
\langle T_{krsl}, h_{kl}h_{rs} \rangle = 2 \mathrm{Re}(\langle T_{k\bar{r}s\bar{l}} ,h_{k\bar{l}}h_{s\bar{r}}\rangle). \label{RiemCom_3}
\end{eqnarray}

\end{lemma}
\begin{proof}
We might as well assume that the basis $v_{1},Jv_{1},v_{2},Jv_{2},\ldots, v_{n},Jv_{n}$ diagonalises the symmetric tensor $h_{kl}$. If
\[
h(v_{k},v_{k}) = h(Jv_{k},Jv_{k}) =  2\lambda_{k},
\]
then
\[
h_{\mathbb{C}}(e_{k},\bar{e}_{k}) = h_{k\bar{k}} = \lambda_{k}.
\]
Equation (\ref{RiemCom_1}) follows from
\[
\langle h_{kr}h^{r}_{l},h_{kl} \rangle = 16\sum_{k=1}^{n}\lambda_{k}^{3} \qquad \mathrm{and} \qquad \langle h_{k\bar{r}}h^{\bar{r}}_{\bar{l}},h_{k\bar{l}} \rangle = 8\sum_{k=1}^{n}\lambda_{k}^{3}.
\]
For Equation (\ref{RiemCom_2}), we note
\[
\langle T_{klrs},h_{kl}h_{rs} \rangle  = 8\sum_{1\leq k,r\leq n}\left(T(v_{k},v_{k},v_{r},v_{r})+T(Jv_{k},Jv_{k},v_{r},v_{r})\right)\lambda_{k}\lambda_{r},
\]
and
\[
\langle T_{k\bar{l}r\bar{s}},h_{k\bar{l}}h_{r\bar{s}} \rangle  = 16\sum_{1\leq k,r\leq n}T(e_{k},\bar{e}_{k},e_{r},\bar{e}_{r})\lambda_{k}\lambda_{r}.
\]
Expanding 
\[
T(e_{k},\bar{e}_{k},e_{r},\bar{e}_{r}) = \frac{1}{16}T(v_{k}-\sqrt{-1}Jv_{k},v_{k}+\sqrt{-1}Jv_{k},v_{r}-\sqrt{-1}Jv_{r},v_{r}+\sqrt{-1}Jv_{r}),
\]
we see the real part is contributed to by taking $0, 2,$ or $4$ terms with $\sqrt{-1}$.  The terms of the form $T(-,-,v_{r},Jv_{r})=T(-,-,Jv_{r},v_{r})=0$ and so we obtain
\[
\langle T_{k\bar{l}r\bar{s}},h_{k\bar{l}}h_{r\bar{s}} \rangle  = 2\sum_{1\leq k,r\leq n}\left(T(v_{k},v_{k},v_{r},v_{r})+T(Jv_{k},Jv_{k},v_{r},v_{r})\right)\lambda_{k}\lambda_{r}.
\] 
For Equation (\ref{RiemCom_3}) \vspace{5pt} \\ 
$\langle T_{krsl},h_{kl}h_{rs} \rangle = $
\begin{small}
\begin{equation*}
 4\sum_{1\leq k,r\leq n}\left(T(v_{k},v_{r},v_{r},v_{k})+T(Jv_{k},v_{r},v_{r},Jv_{k})+T(v_{k},Jv_{r},Jv_{r},v_{k})+T(Jv_{k},Jv_{r},Jv_{r},Jv_{k}) \right)\lambda_{k}\lambda_{r},
\end{equation*}
\end{small}
and 
\[
\langle T_{k\bar{r}s\bar{l}} ,h_{k\bar{l}}h_{s\bar{r}}\rangle = 16\sum_{1\leq k,r\leq n}T(e_{k},\bar{e}_{r},e_{r},\bar{e}_{k})\lambda_{k}\lambda_{r}.
\]
Choosing $0,2,$ or $4$ terms with $\sqrt{-1}$ from
\[
T(e_{k},\bar{e}_{r},e_{r},\bar{e}_{k}) = \frac{1}{16}T(v_{k}-\sqrt{-1}Jv_{k},v_{r}+\sqrt{-1}Jv_{r},v_{r}-\sqrt{-1}Jv_{r},v_{k}+\sqrt{-1}Jv_{k}),
\]
we pick up terms
\[
T(v_{k},v_{r},v_{r},v_{k}),\quad T(v_{k},v_{r},Jv_{r},Jv_{k}), \quad T(Jv_{k},Jv_{r},v_{r},v_{k}), \quad T(Jv_{k},Jv_{r},Jv_{r},Jv_{k}),
\]
\[
T(Jv_{k},v_{r},v_{r},Jv_{k}),\quad -T(Jv_{k},v_{r},Jv_{r},v_{k}),\quad T(v_{k},Jv_{r},Jv_{r},v_{k}),\quad -T(v_{k},Jv_{r},v_{r},Jv_{k}).
\]
Hence \vspace{5pt} \\
$\langle T_{k\bar{r}s\bar{l}} ,h_{k\bar{l}}h_{s\bar{r}}\rangle =$
\begin{small}
\[
 2\sum_{1\leq k,r\leq n}\left(T(v_{k},v_{r},v_{r},v_{k})+T(Jv_{k},v_{r},v_{r},Jv_{k})+T(v_{k},Jv_{r},Jv_{r},v_{k})+T(Jv_{k},Jv_{r},Jv_{r},Jv_{k}) \right)\lambda_{k}\lambda_{r},
\]
\end{small}
and the result follows.
\end{proof}
\subsection{Strategy for demonstrating the rigidity of Grassmannians}
We follow the strategy outlined in \cite{BHMW} where the rigidity of $\mathcal{M}=SU_{2n+1}$ was demonstrated by proving none of the EIDs that the biinvariant metric admits is integrable to second order. (We should say here that this strategy is implicit in Koiso's original work on rigidity). For a general symmetric space $\mathcal{M} = G/K$, the projection of $\mathcal{E}''(h,h)$ to $\varepsilon(g)$ can been seen as an element of $\mathrm{Hom}_{G}(s^{2}(\mathfrak{g}),\mathfrak{g})$ and when $G$ is the special unitary group this Hom-space is one-dimensional. Thus the obstruction map is a multiple of a particular generator (see the discussion in Section 3 of \cite{BHMW}).  In the case of $G=SU_{2n+1}$, this generator does not have any zeros, so, providing the multiple is non-zero, one can conclude all EIDs are obstructed to second order. To show the multiple is non-zero,  we choose carefully a single element $\gamma \in \mathfrak{su}_{n+m}$, form the corresponding EID,  $\mathcal{D}_{\gamma}$, and compute $\mathcal{I}(D_{\gamma})$. We show 
$\mathcal{I}(D_{\gamma}) \neq 0$ which proves Theorem \ref{Thm:Main}

\section{An explicit description of the variations of the Grassmannian}\label{Sec:3}
\subsection{Tautological bundles for the Grassmannian}
We denote by $G_{m}(\mathbb{C}^{n+m})$ the Grassmannian of $m$-planes in $\mathbb{C}^{n+m}$.  The generalised Euler sequence 
\[
0\rightarrow \mathcal{U}\rightarrow \underline{\mathbb{C}}^{n+m}\rightarrow \mathcal{Q}\rightarrow 0,
\]
relates the the trivial  $\mathbb{C}^{n+m}$-bundle over the Grassmannian to its subbundle $\mathcal{U}$, the tautological $m$-plane bundle, and the quotient bundle $\mathcal{Q}$. The holomorphic tangent bundle $\mathcal{T}G_{m}(\mathbb{C}^{n+m})$ is isomorphic to ${\mathrm{Hom}_{\mathbb{C}}(\mathcal{U}, \mathcal{Q})}$. Fixing an Hermitian inner product on the ambient copy of $\mathbb{C}^{n+m}$ induces a fixed Hermitian metric on $\underline{\mathbb{C}}^{n+m}$ and gives a splitting of the bundle $\underline{\mathbb{C}}^{n+m} = \mathcal{U}\oplus\mathcal{U}^{\perp}$ where $\mathcal{Q}\cong \mathcal{U}^{\perp}$ as complex (but not holomorphic) vector bundles. The metric on $\underline{\mathbb{C}}^{n+m} $ restricts to a metric on the subbundles $ \mathcal{U}$ and $\mathcal{U}^{\perp}$ and therefore induces a Hermitian metric on $\mathcal{U}^{\ast}\otimes \mathcal{U}^{\perp}\cong \mathcal{T}G_{m}({\mathbb{C}}^{n+m} )$; this metric, as we shall see, is a K\"ahler--Einstein metric.\\
\\
If, instead of a Hermitian metric on the trivial bundle $\underline{\mathbb{C}}^{n+m}$, we endow it with a general sesquilinear form, restriction also yields sesquilinear forms on the subbbundles $ \mathcal{U}$ and $\mathcal{U}^{\perp}$. Using the Hermitian metric on $\mathcal{U}$, we can induce a sesquilinear form on $\mathcal{U}^{\ast}$. \\
\\
Finally, we note that a sesquilinear forms $P_{1}$ and $P_{2}$ on $\mathcal{U}^{\ast}$ and $\mathcal{U}^{\perp}$ induce a sesquilinear form $P_{1}\cdot P_{2}$ on $\mathcal{T}G_{m}({\mathbb{C}}^{n+m} )$. 

\subsection{Trivialising the tautological vector bundles}
An $m$-plane in $\mathbb{C}^{n+m}$ is the image of a injective linear map with domain $\mathbb{C}^{m}$ i.e. the image of an $(n+m)\times m$ rank $m$ matrix. We will consider the dense open set of planes that can be realised, after multiplication by $GL(m,\mathbb{C})$, by the image of the matrix
\[
M=\left( 
\begin{array}{c}
\mathbb{I}_{m}\\
W
\end{array}
\right),
\]
where $W\in \mathrm{Mat}^{n\times m}(\mathbb{C})$ and $\mathbb{I}_{m}$ is the $m\times m$ identity matrix. The entries in $W$ are complex coordinates for the open set of planes. Accordingly, we will consider a generic coordinate $w_{I}$ where $I =(i_{1},i_{2})$, ${i_{1}\in \{1,2,\ldots, n\}}$, and ${i_{2}\in\{1,2,\ldots ,m\} }$.\\
\\
The $m$ columns of the matrix $M$ can be thought of as sections that generate at each point the tautological rank $m$ bundle $\mathcal{U}\rightarrow G_{m}(\mathbb{C}^{n+m})$.  If we endow the ambient $\mathbb{C}^{n+m}$ with the standard Euclidean Hermitian metric $\langle\cdot,\cdot \rangle_{\mathrm{Euc}}$, then the induced metric on the trivial bundle yields a splitting 
\[
\mathcal{U} = \mathrm{span}\{ c_{1},c_{2},\ldots ,c_{m}\} \quad \mathrm{and} \quad \mathcal{U}^{\perp} = \mathrm{span}\{  q_{1},q_{2},\ldots, q_{n}\},
\]
where 
\[
c_{j} = \left(
\begin{array}{c} 
0\\
\vdots \\
0\\
1\\
0\\
\vdots\\
w_{1j}\\
w_{2j}\\
\vdots\\
w_{nj}
\end{array}
\right) \quad \mathrm{and} \quad q_{j} = \left(
\begin{array}{c} 
-\bar{w}_{j1}\\
-\bar{w}_{j2}\\
\vdots \\
-\bar{w}_{jm}\\
0\\
\vdots \\
0\\
1\\
0\\
\vdots \\
0
\end{array}
\right)
\] 
where the `$1$' is in the $j^{th}$ row for $c_{j}$ and in the in the $(m+j)^{th}$ row for $q_{j}$.
\subsection{Induced metrics on the canonical bundles and the Fubini--Study metric}
With respect to the local frames $\{c_{j}\}$ and $\{q_{j}\}$, the induced Hermitian metrics are given by by
\[
U_{ij} := \langle c_{i}, c_{j} \rangle_{\mathrm{Euc}} \quad \mathrm{and} \quad Q_{ij} := \langle q_{i},q_{j}\rangle_{\mathrm{Euc}}.
\] 
We denote by $U^{-1}$ and $Q^{-1}$ the inverse matrices of $U$ and $Q$ so, for example, ${U^{-1}_{ik}U_{kj} = \delta_{ij}}$.  We record here some useful identities regarding these Hermitian metrics, leaving the proofs to the reader.
\begin{lemma}[Zeroth-order identities]
	Let $U$ and $Q$ be defined as previously. Then the following identities hold:
	\begin{eqnarray}
	U_{ij} = \delta_{ij} + w_{ri}\overline{w}_{rj},\label{U_eqn}\\
	Q_{ij} = \delta_{ij} + \overline{w}_{ir}w_{jr}, \label{Q_eqn}\\
	Q^{-1}_{ij} = \delta_{ij}-(U^{-1}_{rs})\overline{w}_{ir}w_{js},\\
	(Q^{-1}_{ir})\overline{w}_{rj} = (U^{-1}_{rj})\overline{w}_{ir} \label{shift_eqn}.
	\end{eqnarray}
\end{lemma}
We will also need to compute derivatives of the metric quantities in order to work with objects such as connections. 
\begin{lemma}[Derivatives of Hermitian metrics] \label{U_Q_norm}
	Let $U$ and $Q$ be defined as previously. Then the following identities hold:
	\begin{eqnarray}
	\frac{\partial U_{jk}^{-1}}{\partial w_{I}} = -(U^{-1}_{ji_{2}})(U^{-1}_{rk})\overline{w}_{i_{1}r} = -(U^{-1}_{ji_{2}})(Q^{-1}_{i_{1}r})\overline{w}_{rk}, \label{dU} \\
	\frac{\partial Q_{jk}^{-1}}{\partial w_{I}} = -(U^{-1}_{ri_{2}})(Q^{-1}_{i_{1}k})\overline{w}_{jr}, \label{dQ} \\
	\frac{\partial^{2} U_{kl}^{-1}}{\partial w_{I}\partial\overline{w}_{J}} = (U^{-1}_{j_{2}i_{2}})(U^{-1}_{sl})(U^{-1}_{kr})\overline{w}_{i_{1}s}w_{j_{1}r}-(U^{-1}_{k{i_{2}}})(U^{-1}_{j_{2}l})(Q^{-1}_{i_{1}j_{1}}),\\
	\frac{\partial^{2} U_{kl}^{-1}}{\partial w_{I}\partial\overline{w}_{J}} = (U^{-1}_{j_{2}i_{2}})(Q^{-1}_{i_{1}s})(U^{-1}_{kr})\overline{w}_{sl}w_{j_{1}r}-(U^{-1}_{k{i_{2}}})(U^{-1}_{j_{2}l})(Q^{-1}_{i_{1}j_{1}}),\\
	\frac{\partial^{2} Q_{kl}^{-1}}{\partial w_{I}\partial\overline{w}_{J}} = (U^{-1}_{j_{2}r})(U^{-1}_{si_{2}})(Q^{-1}_{i_{1}j_{1}})\overline{w}_{ks}w_{lr}-(U^{-1}_{j_{2}i_{2}})(Q^{-1}_{kj_{1}})(Q^{-1}_{i_{1}l}).
	\end{eqnarray}
\end{lemma}
An immediate consequence of Equations (\ref{dU}) and (\ref{dQ}) is that the complex coordinates $W$ are `normal' coordinates for the Hermitian metrics (in the sense that at $W=0$, both matrices are the identity, and the first derivatives of the metric vanish at $W=0$ too). We will use the following standard fact. 
\begin{lemma}[Generalised Euler Sequence] \label{Lem:Eul_seq}
	The holomorphic tangent bundle of the Grassmannian $G_{m}(\mathbb{C}^{n+m})$ satisfies
	\[
	TG_{m}(\mathbb{C}^{n+m}) \cong \mathrm{Hom}_{\mathbb{C}}(\mathcal{U},\mathcal{Q}) \cong \mathcal{U}^{\ast}\otimes \mathcal{Q}.
	\]
	Furthermore, identifying $\mathcal{Q}\cong \mathcal{U}^{\perp}$ as complex vector bundles, in the local coordinates $W$ we have
	\[
	\frac{\partial}{\partial w_{I}} = c^{\ast}_{i_{2}}\otimes (Q^{-1}_{i_{1}k})q_{k},
	\]
	where $\{c^{\ast}_{j}\}$ denotes the frame of $\mathcal{U}^{\ast}$ dual to $\{c_{j}\}.$
\end{lemma}
Given a Hermitian vector bundle $(E,h_{E})$, the Hermitian metric, $h_{E}$ induces a Hermitian metric $h_{E^{\ast}}$ on the dual bundle $E^{\ast}$. The salient point of the construction is that if, in a local frame $\{e_{i}\}$, the metric $h_{E}$ is described by the Hermitian matrix $(h_{E})_{ij}$, then the induced metric described in the dual frame $\{e_{i}^{\ast}\}$ satisfies
\[
h_{E^{\ast}}(e_{i}^{\ast}, e^{\ast}_{j}) = (h_{E})^{-1}_{ji}.
\] 
The induced metrics on $\mathcal{U}^{\ast}$ and $\mathcal{U}^{\perp}\cong\mathcal{Q}$ (which we denote $g_{1}$ and $g_{2}$ respectively) yield a metric $g=g_{1}\cdot g_{2}$ on the tensor product $\mathcal{U}^{\ast}\otimes \mathcal{Q}$ which, by Lemma \ref{Lem:Eul_seq}, gives rise to a Hermitian metric on the holomorphic tangent bundle; this metric is the {\it Fubini--Study} metric and we now detail some of its interesting properties. Again, we leave the proofs to the reader. 

\begin{lemma}[Properties of Fubini--Study metric]
	In the local holomorphic cordinates $W$, the metric induced on the holomorphic tangent bundle as given by
	\[
	g_{I\bar{J}} = (U^{-1}_{j_{2}i_{2}})(Q^{-1}_{i_{1}j_{1}}).
	\]
	The inverse metric $g^{I\bar{J}}$ (in the sense that $g^{I\bar{K}}g_{J\bar{K}} = \delta_{IJ}$ ) is given by
	\[
	g^{I\bar{J}} =(U_{i_{2}j_{2}})(Q_{j_{1}i_{1}}).
	\]
	The metric is K\"ahler as 
	\[
	\frac{\partial g_{I\bar{J}}}{\partial w_{K}} = \frac{\partial g_{K\bar{J}}}{\partial w_{I}}.
	\]
	The Christoffel symbols are given by
	\[
	\Gamma_{IJ}^{K} = -(\delta_{j_{1}k_{1}})(\delta_{i_{2}k_{2}})(U^{-1}_{sj_{2}})\overline{w}_{i_{1}s} - (\delta_{i_{1}k_{1}})(\delta_{j_{2}k_{2}})(U^{-1}_{si_{2}})\overline{w}_{j_{1}s}. 
	\]
	The metric is K\"ahler--Einstein and satisfies
	\[
	\mathrm{Ric}_{I\bar{J}} = (n+m)g_{I\bar{J}}.
	\]
\end{lemma}
The coordinates $W$ are local {\it normal} holomorphic coordinates for the metric $g$.

\subsection{Seqsquilinear forms and the action of $SU_{n+m}$}
Let ${\gamma \in \sqrt{-1}\mathfrak{su}_{n+m}}$ (here we really just mean $\gamma$ is a trace-free Hermitian matrix) and define the sesquilinear form $P_{\gamma}$ on the ambient vector space $\mathbb{C}^{n+m}$ by
\[
P_{\gamma}(v_{1},v_{2}) :=v_{2}^{\ast}\gamma v_{1},
\]
where $v_{1},v_{2}\in \mathbb{C}^{n+m}$. This induces a sesquilinear form on the trivial $\mathbb{C}^{n+m}$-bundle which we also denote $P_{\gamma}$. There are a number of related objects that $P_{\gamma}$ can be used to create.\\ 
\\
Given a local orthonormal frame of the tautological subbundle $\mathcal{U}$,  $\{\eta_{j}\}$ say, we form the (locally defined) function
\begin{equation}\label{f_def}
f_{\gamma}(x) := \sum_{i=1}^{m}P_{\gamma}(\eta_{i}(x),\eta_{i}(x)).
\end{equation}
We shall see in the next subsection that $f_{\gamma}$ is a actually a globally defined function on the Grassmannian and is in fact an eigenfunction for the Laplacian.
\\
The form $P_{\gamma}$ induces forms on $\mathcal{U}^{\ast}$ and $\mathcal{U}^{\perp}$ which we denote $H_{1}(f_{\gamma})$ and $H_{2}(f_{\gamma})$ respectively.  We can then produce sesquilinear forms $h_{1}$ and $h_{2}$ on ${TG_{m}(\mathbb{C}^{n+m})\cong \mathcal{U}^{\ast} \otimes \mathcal{U}^{\perp}}$ by
\begin{equation}\label{deform_def}
h_{1}=H_{1}(f_{\gamma})\cdot g_{2} \qquad \mathrm{and} \qquad h_{2} =  g_{1} \cdot H_{2}(f_{\gamma}).
\end{equation}
There is a transitive action of $SU_{n+m}$ on $Gr_{m}(\mathbb{C}^{n+m})$ which lifts to an action on the bundles $\mathcal{U}$ and $\mathcal{U^{\perp}}$.  The seqsquilinear forms $h_{1}$ and $h_{2}$ are not invariant under this action but transform under the action to the form associated with $p^{-1}\gamma p$.  Hence, to prove various identities, we need only demonstrate their validity at $W=0$ for an arbitrary $\gamma$.
\subsection{Eigenfunctions of the Laplacian}

\begin{lemma}[Properties of $f_{\gamma}$]
	Let $f_{\gamma}$ be the function defined by Equation (\ref{f_def}). Then $f_{\gamma}$ has the following properties:
	\begin{enumerate}[(i)]
		\item The function $f_{\gamma}$ is independent of the orthonormal frame $\{\eta_{j}\}$ and thus is a globally defined function $f_{\gamma}:G_{m}(\mathbb{C}^{n+m})\rightarrow \mathbb{R}$.
		\item If we denote by $M^{\mathcal{U}}$ the matrix of functions defined by $M^{\mathcal{U}}_{ij}:=P_{\gamma}(c_{i},c_{j})$, then
		\[
		f_{\gamma} = (U^{-1}_{ij})M^{\mathcal{U}}_{ji}.
		\]
		\item  If we denote by $M^{\mathcal{Q}}$ the matrix of functions defined by $M^{\mathcal{Q}}_{ij}:=P_{\gamma}(q_{i},q_{j})$, then
		\[
		f_{\gamma} = -(Q^{-1}_{ij})M^{\mathcal{Q}}_{ji}.
		\]
		\item The function $f_{\gamma}$ is an eigenfunction of the Laplacian with eigenvalue $2(n+m)$
	\end{enumerate}
\end{lemma}
\begin{proof}
	Properties (i) and (ii) are easily established. For (iii) we note that $\gamma$ is trace-free and 
	\[
	(U_{lk}^{-1})M_{kl}^{\mathcal{U}}+(Q^{-1}_{rs})M^{\mathcal{Q}}_{sr} = \mathrm{tr}(\gamma)=0.
	\]		
	To show (vi), we compute at $W=0$ and note $f_{\gamma}(0) = \sum_{k=1}^{m}\gamma_{kk}$. We also have 
	\[
	g^{I\bar{J}}\frac{\partial^{2}f_{\gamma}}{\partial w_{I} \partial \overline{w}_{J}}\bigg |_{W=0} = g^{I\bar{J}}\left(\frac{\partial^{2}(U^{-1}_{kl})}{\partial w_{I} \partial \overline{w}_{J}}M^{\mathcal{U}}_{lk}+ \frac{\partial (U^{-1}_{kl})}{\partial w_{I}}\frac{\partial M^{\mathcal{U}}_{lk} }{\partial \overline{w}_{J}} + \frac{\partial (U^{-1}_{kl})}{\partial \overline{w}_{J}}\frac{\partial M^{\mathcal{U}}_{lk} }{\partial {w}_{I}} +(U^{-1}_{kl}) \frac{\partial^{2}M^{\mathcal{U}}_{lk}}{\partial w_{I} \partial \overline{w}_{J}}\right)\bigg |_{W=0}
	\]
	\[
	=\sum_{I}\left(\frac{\partial^{2}(U^{-1}_{kl})}{\partial w_{I} \partial \overline{w}_{I}}M^{\mathcal{U}}_{lk}+\sum_{k=1}^{m} \frac{\partial^{2}(M^{\mathcal{U}}_{kk})}{\partial w_{I} \partial \overline{w}_{I}}\right)\bigg |_{W=0} = -n\sum_{k=1}^{m}\gamma_{kk}+m\sum_{k=1}^{n}\gamma_{(m+k)(m+k)}.
	\]
	The result follows from the fact $\gamma$ is trace-free and noting that for any function $\psi$ on a K\"ahler manifold $(\mathcal{M},g)$ with local holomorphic coordinates $z_{i}$,
	\[
	g^{k\bar{l}}\frac{\partial^{2}\psi}{\partial z_{k}\partial \bar{z}_{l}} = -\frac{1}{2}\Delta \psi.
	\]
\end{proof}
\subsection{The deformations}
We compute various quantities associated to the two sesquilinear forms given by Equation (\ref{deform_def}).  This will allow us to find transverse trace-free tensors. It is straightforward to see that in the local coordinates $W$, the tensor fields are given by
\[
(h_{1})_{I\bar{J}} = (U^{-1}_{j_{2}i_{2}})(Q^{-1}_{i_{1}k})(Q^{-1}_{lj_{1}})M^{\mathcal{Q}}_{kl}
\]
and
\[
(h_{2})_{I\bar{J}} = (U^{-1}_{j_{2}k})(U^{-1}_{li_{2}})M^{\mathcal{U}}_{kl}Q^{-1}_{i_{1}j_{1}}. 
\]
We begin by proving a result due to Gasqui and Goldschmidt relating the tensors $h_{1}$ and $h_{2}$ to the Hessian of the eigenfunction $f_{\gamma}$. The proof we give here is an elementary calculation in local normal coordinates and is very different from the one given in \cite{GGbook}.

\begin{lemma}[Gasqui--Goldschmidt, cf. Propsition 8.5 in \cite{GGbook}] \label{GG_Lem}
	Let $f_{\gamma}$ be given by Equation (\ref{f_def}) and let $h_{1}$ and $h_{2}$ be as in Equation (\ref{deform_def}). Then
	\begin{equation} \label{GG_eqn}
	\mathrm{Hess}(f_{\gamma}) = h_{1}-h_{2}.
	\end{equation}
\end{lemma}	
\begin{proof}
	As $f_{\gamma}$ is an eigenfunction with eigenvalue twice the Einstein constant, Matsushima's Theorem \cite{Mat} means that $J\nabla f_{\gamma}$ is a holomorphic vector field. It follows that the Hessian of $f_{\gamma}$ has no $J$-anti-invariant component and so we are done if we can establish in local coordinates
	\[
	\frac{\partial^{2}f_{\gamma}}{\partial w_{I}\partial\bar{w}_{J}} = (h_{1})_{I\bar{J}}-(h_{2})_{I\bar{J}}
	\]
	We compute at $W=0$ and we have already seen 
	\[
	\frac{\partial^{2}f_{\gamma}}{\partial w_{I} \partial \overline{w}_{J}} \bigg |_{W=0} = \left( \frac{\partial^{2}(U^{-1}_{kl})}{\partial w_{I} \partial \overline{w}_{J}}M^{\mathcal{U}}_{lk}+(U^{-1}_{kl}) \frac{\partial^{2}M^{\mathcal{U}}_{lk}}{\partial w_{I} \partial \overline{w}_{J}}\right) \bigg |_{W=0} =  -\delta_{i_{1}j_{1}}\gamma_{i_{2}j_{2}}+\delta_{i_{2}j_{2}}\gamma_{(j_{1}+m)(i_{1}+m)}.
	\]
	Using the definition of $h_{1}$ and $h_{2}$, we calculate
	\[
	(h_{1}-h_{2})_{I\bar{J}} = \delta_{i_{2}j_{2}}M^{Q}_{i_{1}j_{1}}-\delta_{i_{1}j_{1}}M^{\mathcal{U}}_{j_{2}i_{2}}=-\delta_{i_{1}j_{1}}\gamma_{i_{2}j_{2}} + \delta_{i_{2}j_{2}}\gamma_{(j_{1}+m)(i_{1}+m)}.
	\]
\end{proof}

\begin{lemma}\label{Lem_div}
	Let $f_{\gamma}$ be given by Equation (\ref{f_def}) and let $h_{1}$ and $h_{2}$ be as in Equation (\ref{deform_def}). Then  the traces are given by
	\[
	g^{I\bar{J}}(h_{1})_{I\bar{J}} = -mf_{\gamma}  \qquad \mathrm{and} \qquad  g^{I\bar{J}}(h_{2})_{I\bar{J}} = nf_{\gamma}.
	\]
	If we denote by $\sigma_{1}$ and $\sigma_{2}$ the  $(1,1)$-forms associated to $h_{1}$ and $h_{2}$ respectively, then
	\[
	\overline{\partial}^{\ast}\sigma_{1} =-n \sqrt{-1}\partial f_{\gamma}  \qquad \mathrm{and} \qquad 	\overline{\partial}^{\ast}\sigma_{2} =m\sqrt{-1}\partial f_{\gamma}.
	\]
\end{lemma}
\begin{proof}
	The statements about traces are straightforward. For the statement about the codifferential $\overline{\partial}^{\ast}$ we compute in the coordinates at $W=0$.\\
	We note that $(\partial f_{\gamma})_{I} = \gamma_{i_{2}(i_{1}+m)}$. The associated form $\sigma_{2}$ is given by
	\[
	\sigma_{2} = \sqrt{-1}(h_{2})_{I\bar{J}} \ dw_{I}\wedge d\overline{w}_{{J}},
	\]
	and thus
	\[
	(\overline{\partial}^{\ast}\sigma_{2})_{I} = \sqrt{-1}g^{K\bar{J}}\nabla_{K}(h_{2})_{I\bar{J}}= \sqrt{-1} \frac{\partial(h_{2})_{I\bar{J}}}{\partial w_{J}} = \sqrt{-1}\delta_{j_{2}k}\delta_{li_{2}}\delta_{kj_{2}}\gamma_{l(j_{1}+m)}\delta_{i_{1}j_{1}}
	\]
	\[
	=m\sqrt{-1}\gamma_{i_{2}(i_{1}+m)}.
	\]
	With this in hand, we can appeal to the Gasqui--Goldschmidt Lemma \ref{GG_Lem}. We have
	\[
	-\sqrt{-1}\partial\Delta_{\overline{\partial}}f=\overline{\partial}^{\ast}\sqrt{-1}\partial\overline{\partial}f_{\gamma} = \overline{\partial}^{\ast}\sigma_{1} - \overline{\partial}^{\ast}\sigma_{2} = \overline{\partial}^{\ast}\sigma_{1} -m\sqrt{-1}\partial f_{\gamma}.
	\]
	Hence 
	\[
	-(n+m)\sqrt{-1}\partial f_{\gamma} = \overline{\partial}^{\ast}\sigma_{1}-m\sqrt{-1}\partial f_{\gamma},
	\]
	and the result follows.
\end{proof}
The final quantity we need to compute is the $\overline{\partial}$-Laplacian applied to the forms $\sigma_{1}$ and $\sigma_{2}$.
\begin{lemma} \label{Lem_Lap}
	Let $\sigma_{1}$ and $\sigma_{2}$ be the $(1,1)$-forms associated to the tensors $h_{1}$ and $h_{2}$ respectively. Then
	\[
	\Delta_{\overline{\partial}}\ \sigma_{1}=(n+m)\sigma_{1} \qquad \mathrm{and} \qquad \Delta_{\overline{\partial}} \ \sigma_{2}=(n+m)\sigma_{2}.
	\]
\end{lemma}
\begin{proof}
	We begin by computing $\Delta_{\bar{\partial}}\sigma_{2}$.  Using Lemma \ref{Lem_div}, we have
	\[
	\Delta_{\bar{\partial}} \ \sigma_{2}= \left(\bar{\partial}\bar{\partial}^{\ast}+\bar{\partial}^{\ast}\bar{\partial}\right)\sigma_{2} = -m\sqrt{-1}\partial\bar{\partial} f_{\gamma}+ \bar{\partial}^{\ast}\bar{\partial}\sigma_{2}.
	\]
	At $W=0$, 
	\[
	(\sqrt{-1}\partial\bar{\partial} f_{\gamma})_{I\bar{J}} = \sqrt{-1}\left(-\delta_{i_{1}j_{1}}\gamma_{i_{2}j_{2}}+\delta_{i_{2}j_{2}}\gamma_{(j_{1}+m)(i_{1}+m)}\right).
	\]
	To calculate $\bar{\partial}^{\ast}\bar{\partial}\sigma_{2}$, we begin by noting that for a general $(1,2)$-form,
	\[
	T= T_{\bar{k}i\overline{l}} \ d\bar{z}_{k}\wedge dz_{i}  \wedge d \bar{z}_{l},
	\]
	we have, in holomorphic normal coordinates at $z=0$, 
	\[
	(\overline{\partial}^{\ast}T)_{r\overline{s}} = -\left(\frac{\partial T_{\bar{k}r\bar{s}}}{\partial z_{k}}-\frac{\partial T_{\bar{s}r\bar{k}}}{\partial z_{k}}\right).
	\]
	Hence we have, at $W=0$, 
	\[
	\left(\bar{\partial}^{\ast}\bar{\partial}\sigma_{2}\right)_{I\bar{J}} =-\sqrt{-1}\left(\frac{\partial^{2}(h_{2})_{I\bar{J}}}{\partial{w}_{K}\partial\bar{w}_{K}} -\frac{\partial^{2}(h_{2})_{I\bar{K}}}{\partial w_{K}\partial\bar{w}_{J}}  \right) \bigg |_{W=0}
	\]
	Using the expression for $h_{2}$ in local coordinates and the fact that from Lemma \ref{U_Q_norm}, single derivatives of the quantities $U^{-1}_{rs}$ and $Q^{-1}_{rs}$ vanish at $W=0$, we have
	\[
	\frac{\partial^{2}(h_{2})_{I\bar{J}}}{\partial{w}_{K}\partial\bar{w}_{K}}\bigg |_{W=0}  = -(2n+m)\delta_{i_{1}j_{1}}\gamma_{i_{2}j_{2}}+\delta_{i_{1}j_{1}}\delta_{i_{2}j_{2}}\sum_{r=1}^{n}\gamma_{(r+m)(r+m)}
	\]
	and 
	\[
	\frac{\partial^{2}(h_{2})_{I\bar{K}}}{\partial w_{K}\partial\bar{w}_{J}}\bigg |_{W=0}  = -(n+m)\delta_{i_{1}j_{1}}\gamma_{i_{2}j_{2}}-\delta_{i_{1}j_{1}}\delta_{i_{2}j_{2}}\sum_{k}\gamma_{kk}+m\delta_{i_{2}j_{2}}\gamma_{(j_{1}+m)(i_{1}+m)}
	\]
	Thus, at $W=0$,
	\[
	\left(\bar{\partial}^{\ast}\bar{\partial}\sigma_{2}\right)_{I\bar{J}} = n\delta_{i_{1}j_{1}}\gamma_{i_{2}j_{2}}+m\delta_{i_{2}j_{2}}\gamma_{(j_{1}+m)(i_{1}+m)}.
	\]
	So
	\[
	(\Delta_{\bar{\partial}} \ \sigma_{2})_{I\bar{J}} = (n+m)\delta_{i_{1}j_{1}}\gamma_{i_{2}j_{2}}= (n+m)(\sigma_{2})_{I\bar{J}}.
	\]
	The result for $\sigma_{1}$ follows from the Gasqui--Goldschmidt Lemma \ref{GG_Lem}.
\end{proof}
\begin{theorem}[Infinitesimal deformations of the Fubini--Study metric] \label{Thm:deform}
	Let $h_{1}$ and $h_{2}$ be as in Equation (\ref{deform_def}) and let $h_{3}=f_{\gamma}g$. Then, if $m,n\geq2$,	
	\[
	\mathcal{D}_{\gamma} := n(1-m^{2})h_{1}+m(1-n^{2})h_{2}+(n^{2}-m^{2})h_{3}, 
	\]
	is an EID for the Fubini--Study metric on $G_{m}(\mathbb{C}^{n+m})$.
\end{theorem}
\begin{proof}
	It is easy to check that $\sigma_{3} = f_{\gamma}\omega$ (the two-form associated to $h_{3}$) satisfies
	\[
	\Lambda(\sigma_{3}) = mnf_{\gamma}, \qquad \bar{\partial}^{\ast}\sigma_{3} = \sqrt{-1}\partial f_{\gamma}, \qquad \mathrm{and} \qquad \Delta_{\bar{\partial}}\sigma_{3}=(n+m)\sigma_{3} 
	\]
	The result follows by combining these calculations with those in Lemma \ref{Lem_div} and Lemma \ref{Lem_Lap}.
\end{proof} 
In the case that $m=1$, it is easy to see that $h_{2}=f_{\gamma}g=h_{3}$ and so $\mathcal{D}_{\gamma}$ vanishes (as expected as it was demonstrated in \cite{KoiOsaka1} that the Fubini--Study metric on $\mathbb{CP}^{n}$ does not admit EIDs). When $m=n$, we also recover Proposition 8.6 from the book \cite{GGbook} which says that $h_{1}+h_{2}$ is an EID for the metric on $G_{n}(\mathbb{C}^{2n})$.
 
\section{Computing Koiso's obstruction}
\subsection{The Theorem of Adler and van Moerbeke}
As we shall see in subsequent subsections, to compute the terms in the Koiso obstruction (\ref{Koiobs}), we will need to consider integrals of the form
\[
\int_{\mathbb{C}^{mn}}\tr(U^{-1})^{\alpha} d\mu(W)  \qquad \mathrm{and} \qquad \int_{\mathbb{C}^{mn}}\tr(U^{-\beta}) d\mu(W),
\]
where $\alpha,\beta \in \{0,1,2,3\}$, $W$ are the coordinates on the dense copy of $\mathbb{C}^{mn}$ contained in $G_{m}(\mathbb{C}^{n+m})$ and $d\mu(W)$ is shorthand for the (possibly rescaled) volume form defined by the Fubini--Study metric. In the local coordinates $W$, 
\[
\tr(U^{-1}) = \tr((\mathbb{I}_{m}+W^{t}\overline{W})^{-1}) = \tr((\mathbb{I}_{m}+W^{\ast}W)^{-1}).
\]

In \cite{AvM}, Adler and van Moerbeke considered integrals of the form
\[
\int_{\mathbb{C}^{mn}}e^{x \tr((\mathbb{I}_{m}+W^{\ast}W)^{-1})}d\mu(W).
\]

What they demonstrate is that one can use the Weyl integration formula to write the previous integral as a multiple of what is known as a generalised Selberg-type integral. Paraphrasing Theorem 1.1 in \cite{AvM} we have
\[
\int_{\mathbb{C}^{mn}}e^{x \tr((\mathbb{I}_{m}+W^{\ast}W)^{-1})}d\mu(W) = 
\int_{[0,1]^{m}}e^{x\sum_{i=1}^{m}z_{i}}\Delta_{m}(z)^{2}\prod_{1}^{m}(1-z_{i})^{n-m} dz_{1}dz_{2}\ldots dz_{m}, 
\]
where $\Delta_{m}(z)$ is the Vandermonde determinant in the $m$ variables $z_{1},z_{2},\ldots, z_{m}$. The $\sum_{i}{z_{i}}$ term corresponds directly to the sum of the eigenvalues of $(\mathbb{I}_{m}+W^{\ast}W)^{-1}$. Hence we obtain directly
\begin{equation} \label{alpha_int}
\int_{\mathbb{C}^{mn}}(\tr(U^{-1}))^{\alpha} d\mu(W) = \int_{[0,1]^{m}}\left(\sum_{i=1}^{m}z_{i}\right)^{\alpha}\Delta_{m}(z)^{2}\prod_{1}^{m}(1-z_{i})^{n-m} dz_{1}dz_{2}\ldots dz_{m}.
\end{equation}
Tracing through the proof of Theorem 1.1 in \cite{AvM}, one also obtains the easy generalisation 
\begin{equation}\label{beta_int}
\int_{\mathbb{C}^{mn}}\tr(U^{-\beta}) d\mu(W) = \int_{[0,1]^{m}}\left(\sum_{i=1}^{m}z_{i}^{\beta}\right)\Delta_{m}(z)^{2}\prod_{1}^{m}(1-z_{i})^{n-m} dz_{1}dz_{2}\ldots dz_{m}.
\end{equation}
The Selberg-type integrals on the right of Equations (\ref{alpha_int}) and (\ref{beta_int}) have been widely studied. We use the results of Kanecko \cite{Kan} (essentially Equation (2.3.13) in \cite{AvM}) which says for a Schur function $s_{\lambda}(z)$ in $m$-variables corresponding to a partition $\lambda=(\lambda_{1},\lambda_{2},\ldots,\lambda_{m})$, and integers $a,b>m-1$,
\begin{small}
\begin{equation*}
\int_{[0,1]^{m}}s_{\lambda}(z)\Delta_{m}(z)^{2}\prod_{1}^{m}(1-z_{i})^{a-m}z_{i}^{b-m} dz = s_{\lambda}(1^{m})\prod_{1}^{m}\frac{\Gamma(i+1)\Gamma(a+1-i)\Gamma(\lambda_{i}+b+1-i)}{\Gamma(\lambda_{i}+a+b+(1-i))},
\end{equation*}
\end{small}
where $1^{m} = (1,1,\ldots,1)$. Evaluating using  $a=n$ and $b=m$ yields
\begin{small}
\begin{equation}
\int_{[0,1]^{m}}s_{\lambda}(z)\Delta_{m}(z)^{2}\prod_{1}^{m}(1-z_{i})^{n-m} dz_{1}dz_{2}\ldots dz_{m}=s_{\lambda}(1^{m})\prod_{1}^{m}\frac{i!(n-i)!(m+\lambda_{i}-i)!}{(n+m+\lambda_{i}-i)!}.
\end{equation}
\end{small}
We fix the constant
\[
C_{n,m} = \prod_{j=1}^{m}\frac{(n-j)!(m-j)!j!}{(n+m-j)!}.
\]
If we abuse notation by ignoring zeros at the end of partitions and write
\[
\int s_{\lambda}  = \int_{[0,1]^{m}}s_{\lambda}(z)\Delta_{m}(z)^{2}\prod_{1}^{m}(1-z_{i})^{n-m} dz_{1}dz_{2}\ldots dz_{m},
\]
we have
\[
C_{n,m}\int s_{(0)} = 1, \qquad C_{n,m}\int s_{(1)} = \frac{m^{2}}{n+m},
\]
\[
C_{n,m}\int s_{(2,0)} = \frac{m^{2}(m+1)^{2}}{2(n+m)(n+m+1)},  \qquad C_{n,m}\int s_{(1,1)} = \frac{m^{2}(m-1)^{2}}{2(n+m)(n+m-1)},
\]
\[
C_{n,m}\int s_{(3,0,0)} = \frac{(m(m+1)(m+2))^{2}}{6(n+m+2)(n+m+1)(n+m)},
\]
\[
C_{n,m}\int s_{(2,1,0)} = \frac{((m-1)m(m+1))^{2}}{3(n+m-1)(n+m)(n+m+1)},
\]
and
\[
C_{n,m}\int s_{(1,1,1)} = \frac{((m-2)(m-1)m)^{2}}{6(n+m-2)(n+m-1)(n+m)}.
\]
Putting this together, we have the following result.
\begin{lemma}\label{Lem:integrals}
Let $F:G_{m}(\mathbb{C}^{n+m})\rightarrow \mathbb{C}$ and write
\[
\int F = C_{n,m}\int_{\mathbb{C}^{mn}}F(W) d\mu(W),
\]
In the notation of the previous section we have
\begin{small}
\[
\int 1 = 1,
\]
\[
\int \mathrm{tr}(U^{-1}) = \frac{m^{2}}{n+m},
\]
\[
\int \mathrm{tr}(U^{-2}) = \int s_{(2,0)}-s_{(1,1)} = \frac{m^{2}(m^2 + 2nm - 1)}{(n+m-1)(n+m)(n+m+1)},
\]
\[
\int \left(\mathrm{tr}(U^{-1})\right)^{2} = \int s_{(2,0)}+s_{(1,1)} = \frac{m^{2}(m^3 + nm^2 - m + n)}{(n+m-1)(n+m)(n+m+1)},
\]
\[
\int \mathrm{tr}(U^{-3}) = \int s_{(3,0,0)}-s_{(2,1,0)}+s_{(1,1,1)} =  \frac{m^{2}(m^4 + 4m^3n + 5m^2n^2 - 5m^2 - 10mn + n^2 + 4)}{(n+m-2)(n+m-1)(n+m)(n+m+1)(n+m+2)},
\]
\[
\int \mathrm{tr}(U^{-2})tr(U^{-1}) =  \int s_{(3,0,0)}-s_{(1,1,1)} =  \frac{m^{2}(m^5 + 3m^4n + 2m^3n^2 - 5m^3 - 5m^2n + 4mn^2 + 4m - 4n)}{(n+m-2)(n+m-1)(n+m)(n+m+1)(n+m+2)},
\]
and
\[
\int \left(\mathrm{tr}(U^{-1})\right)^{3} =  \int s_{(3,0,0)}+2s_{(2,1,0)}+s_{(1,1,1)} = \frac{m^{2}(m^6 + 2m^5n + m^4n^2 - 5m^4 + 3m^2n^2 + 4m^2 - 8mn + 2n^2)}{(n+m-2)(n+m-1)(n+m)(n+m+1)(n+m+2)}.
\]
\end{small}

\end{lemma}
\subsection{Calculating pointwise quantities}Henceforth, all the calculations we make will be in complex coordinates unless specifically mentioned. 
We fix a choice of ${\gamma\in \sqrt{-1}\mathfrak{su}_{n+m}}$ to be 
\begin{equation}\label{Eqn:Gamma_choice}
\gamma=\mathrm{Diag}(\underbrace{-n,-n,\ldots,-n}_{m \ \mathrm{terms}},\underbrace{m,m,\ldots,m}_{n \ \mathrm{terms}}).
\end{equation}
This means
\[
M_{ij}^{\mathcal{U}} = -n\delta_{ij}+mw_{ri}\bar{w}_{rj} = -(n+m)\delta_{ij}+mU_{ij},
\]
where we have used Equation (\ref{U_eqn}). Similarly, using Equation (\ref{Q_eqn}) 
\[
M^{\mathcal{Q}}_{ij} = (n+m)\delta_{ij}-nQ_{ij}.
\]
In local coordinates $w_{I}$ we have
\begin{eqnarray}
f_{\gamma} = U^{-1}_{ij}M^{\mathcal{U}}_{ji} = m^{2}-(n+m)U^{-1}_{ii} =  m^{2}-(n+m)\tr(U^{-1}),\label{eqn:f_gam}\\
(h_{1})_{I\bar{J}} = (U^{-1}_{j_{2}i_{2}})(Q^{-1}_{i_{1}k})(Q^{-1}_{lj_{1}})M^{\mathcal{Q}}_{kl} = -ng_{I\bar{J}}+(n+m)g_{I(r,j_{2})}Q^{-1}_{rj_{1}},\label{eqn:h_1_coord}\\
(h_{2})_{I\bar{J}} = (U^{-1}_{j_{2}k})(U^{-1}_{li_{2}})M^{\mathcal{U}}_{kl}Q^{-1}_{i_{1}j_{1}} = mg_{I\bar{J}}-(n+m)g_{I(j_{1},r)}U^{-1}_{j_{2}r}\label{eqn:h_2_coord}. 
\end{eqnarray}
In the expressions for $h_{1}$ and $h_{2}$ we can see a simplifying calculational principle: 
to switch an $h_{1}$ with an $h_{2}$ in an expression, switch $m$ and $n$, multiply by $-1$, switch $U$ with the conjugate expression in $Q$.    The pointwise norms of various expressions also involve $\tr(Q^{-1})$ terms.  The following identities can be checked directly
\begin{equation}
\tr(Q^{-\alpha}) = (n-m)+\tr(U^{-\alpha}),
\end{equation}
where $\alpha \in \{1,2,3\}$. We record here the following quadratic quantities
\begin{eqnarray}
\| h_{1}\|^{2} = m\left(n^{3}-2n(n+m)\tr(Q^{-1})+(n+m)^{2}\tr(Q^{-2}) \right),\label{quad_1} \\  
\| h_{2}\|^{2} = n\left(m^{3}-2m(n+m)\tr(U^{-1})+(n+m)^{2}\tr(U^{-2})\right),\\
\|h_{3}\|^{2} = mnf_{\gamma}^{2},\\
\langle h_{1},h_{2}\rangle = -f_{\gamma}^{2},\\
\langle h_{1},h_{3}\rangle = -mf_{\gamma}^{2},\\
\langle h_{2},h_{3}\rangle = nf_{\gamma}^{2}. \label{quad_6}
\end{eqnarray}
\subsection{Integral of cubes of eigenfunctions}
As mentioned in Section \ref{Sec:2}, when ${G=SU_{n+m}}$, the space ${\mathrm{Hom}_{G}(s^{2}(\mathfrak{g}),\mathfrak{g})}$ is one-dimensional. Thus the space of $SU_{n+m}$-invariant cubic polynomials on $\mathfrak{su}_{n+m}$ is also one-dimensional and  we can pick any non-vanishing invariant polynomial as a generator; we choose (as do the authors in \cite{SchSem24}) the integral of $f_{\gamma}^{3}$ and express the obstruction (\ref{Koiobs}) as a multiple of this integral. Using the special choice $\gamma$ given in Equation (\ref{Eqn:Gamma_choice}), we have
\[
\int f_{\gamma}^{3} = \int \left(m^{6}-3m^{4}(n+m)\mathrm{tr}(U^{-1})+3m^{2}(n+m)^{2}(\tr(U^{-1}))^{2}-(n+m)^{3}(\mathrm{tr}(U^{-1}))^{3}\right).
\]
We calculate (using Matlab's symbolic algebra toolbox)
\[
\int f_{\gamma}^{3} = \frac{-2m^{2}n^{2}(n-m)^{2}}{(n+m-2)(n+m-1)(n+m+1)(n+m+2)}.
\]

 Note that this polynomial vanishes if $m=n$ but this is not a problem for the `odd' Grassmannians. This integral was calculated using different methods rooted in symplectic geometry in \cite{HMW}; a result of Kr\"oncke \cite{KlKr1} implies that if the integral does not vanish, the Fubini--Study metric on $G_{m}(\mathbb{C}^{n+m})$ is unstable as a fixed point of the Ricci flow. We note, as a sanity check, we do indeed get the same value (compare Lemma 4.6 in \cite{HMW}). Thus fact that the integral is non-zero for $m\neq n$ provides an alternative proof of the dynamic instability of the Fubini--Study metric.
\subsection{The $\langle h_{ik}h^{k}_{j},h_{ij} \rangle $ term}

We collect calculations needed to compute ${\langle h_{ik}h^{k}_{j},h_{ij}\rangle = \tr(H^{3})}$. The EIDs are constructed from combinations of the tensors
$h_{1},h_{2},$ and $h_{3}$.  We denote by 
\begin{equation} \label{EqnZ_defn}
Z_{ijk} :=  \tr(H_{i}H_{j}H_{k})
\end{equation}
where $(H_{i})_{I}^{K} = (h_{i})_{I\bar{K}}g^{J\bar{K}}$. It is immediate that $Z_{ijk} = Z_{jki}=Z_{kij}$ and the observation regarding expressions in $U$ and $Q$ means that we can switch `1' and `2' in calculations by exchanging $n$ and $m$, $U$ and $Q$, and changing a sign if needed. Hence we need only compute
\[
Z_{111}, \quad Z_{112}, \quad Z_{123},\quad Z_{133},\quad Z_{113},\quad \mathrm{and} \quad Z_{333}.  
\]
We have $H_{3} = f_{\gamma}\mathrm{Id}$; combining this with the calculation of the quadratic quantities in Equations (\ref{quad_1})-(\ref{quad_6}) yields
\begin{lemma} \label{Lem:Z_ptwse}
	Let $\gamma$ be as in Equation \ref{Eqn:Gamma_choice} and let $D_{abc}$ be defined by Equation (\ref{EqnZ_defn}). Then
	\begin{small}
	\begin{eqnarray}
	Z_{111} = m(-n^{4}+3n^{2}(n+m)\tr(Q^{-1})-3n(n+m)^{2}\tr(Q^{-2})+(n+m)^{3}\tr(Q^{-3})),\\
	Z_{112} =  (m^2-(n+m)\tr(U^{-1}))(n^{3}-2n(n+m)\tr(Q^{-1})+(n+m)^{2}\tr(Q^{-2})),\\
		Z_{333} = mnf_{\gamma}^{3},\\ 
		Z_{123} = f_{\gamma}\tr(H_{1}H_{2}) = -f^{3}_{\gamma}, \\ Z_{133}=f_{\gamma}\tr(H_{1}) = -mf_{\gamma}^{3},\\
		Z_{113} = f_{\gamma}\|h_{1}\|^{2}.
	\end{eqnarray}
\end{small}
\end{lemma}
We now integrate using Lemma \ref{Lem:integrals} and multiply by a factor of 2 to convert complex to Riemannian coordinates as in Lemma \ref{Lem:ComptoRiem} to compute the first integral in Koiso's obstruction (\ref{Koiobs}). We mention again that these calculations are carried out using the symbolic toolbox in Matlab.
\begin{lemma} \label{Lem:ZOterms}
Let $\gamma$ be as in Equation (\ref{Eqn:Gamma_choice})
and let $\mathcal{D}_{\gamma}$ be the associated EID as described in Theorem \ref{Thm:deform}. Then, in Riemannian coordinates, \\
\begin{small}
$\langle (\mathcal{D}_{\gamma})_{ik}(\mathcal{D}_{\gamma})^{k}_{j},(\mathcal{D}_{\gamma})_{ij} \rangle  = $
\[
-\left[(m^2 - 1)(n^2 - 1)(n+m)^3\frac{(m^3n^3 - 4m^3n - 4mn^3 + 2m^2n^2 + m^2 + n^2 + 7mn - 4)}{(n - m)}\right]\int f^{3}_{\gamma}.
\]
\end{small}
\end{lemma}

\subsection{The $\langle \nabla_{i}\nabla_{j}h_{kl},h_{ij}h_{kl}\rangle$ term}
We return to making calculations in complex coordinates and note that as the EID $\mathcal{D}_{\gamma}$ is divergence-free, integration by parts yields 
\[
\langle \nabla_{I}\nabla_{\bar{J}}(\mathcal{D}_{\gamma})_{K\bar{L}}, (\mathcal{D}_{\gamma})_{I\bar{J}}(\mathcal{D}_{\gamma})_{K\bar{L}} \rangle_{L^{2}} = - \langle \nabla_{\bar{J}}(\mathcal{D}_{\gamma})_{K\bar{L}},(\mathcal{D}_{\gamma})_{I\bar{J}}\nabla^{I}(\mathcal{D}_{\gamma})_{K\bar{L}}\rangle_{L^{2}}. 
\]
Thus it will be necessary to compute terms of the form
\[
 \nabla_{\bar{J}}(h_{r})_{K\bar{L}}\nabla^{I}(h_{s})^{K\bar{L}},
\]
where $r,s\in\{1,2,3\}$.
The formulae for $h_{1}$ and $h_{2}$ in Equations (\ref{eqn:h_1_coord}) and (\ref{eqn:h_2_coord}) yield the following 
\begin{eqnarray*}
\nabla_{J}(h_{1})_{K\bar{L}} = (n+m)g_{K(r,l_{2})}\frac{\partial Q^{-1}_{rl_{1}}}{\partial w_{J}} = -(n+m)g_{J(l_{1},\alpha_{2})}g_{K(\alpha_{1},l_{2})}\overline{w}_{\alpha},\\
\nabla_{J}(h_{2})_{K\bar{L}} =-(n+m)g_{K(l_{1},r)}\frac{\partial U^{-1}_{l_{2}r}}{\partial w_{J}}=(n+m)g_{J(\alpha_{1},l_{2})}g_{K(l_{1},\alpha_{2})}\overline{w}_{\alpha},\\
\nabla_{J}(h_{3})_{K\bar{L}} = (n+m)g_{J\bar{\alpha}}g_{K\bar{L}}\overline{w}_{\alpha}.
\end{eqnarray*}
The nine possible combinations of $r$ and $s$ in
\[
\nabla_{\bar{J}}(h_{r})_{K\bar{L}}\nabla^{I}(h_{s})^{K\bar{L}}
\]
actually yield a very limited set of expressions.
\begin{lemma} \label{Lem:IBPLem1}
Let $\gamma$ be as in Equation (\ref{Eqn:Gamma_choice}) and let $h_{1}, h_{2}, $ and $h_{3}$ be as in Theorem \ref{Thm:deform}. Then
\begin{eqnarray*}
	\left(\nabla_{\bar{J}}(h_{1})_{K\bar{L}}\right)\left(\nabla^{\bar{I}}(h_{1})^{K\bar{L}} \right) =m(n+m)^{2} \delta_{i_{1}j_{1}}\left(U^{-1}_{j_{2}i_{2}}-U^{-1}_{j_{2}\beta}U^{-1}_{\beta i_{2}}\right), \\
\left(\nabla_{\bar{J}}(h_{1})_{K\bar{L}}\right)\left(\nabla^{\bar{I}}(h_{2})^{K\bar{L}} \right)=-(n+m)^{2}g_{K\bar{J}}w_{K}\overline{w}_{I},\\
\left(\nabla_{\bar{J}}(h_{1})_{K\bar{L}}\right)\left(\nabla^{\bar{I}}(h_{3})^{K\bar{L}} \right) =-m(n+m)^{2}g_{K\bar{J}}w_{K}\overline{w}_{I},\\
\left(\nabla_{\bar{J}}(h_{2})_{K\bar{L}}\right)\left(\nabla^{\bar{I}}(h_{1})^{K\bar{L}} \right)= -(n+m)^{2}g_{K\bar{J}}w_{K}\overline{w}_{I},\\
\left(\nabla_{\bar{J}}(h_{2})_{K\bar{L}}\right)\left(\nabla^{\bar{I}}(h_{2})^{K\bar{L}} \right) =n(n+m)^{2}\delta_{i_{2}j_{2}}\left(Q^{-1}_{i_{1}j_{1}}-Q^{-1}_{i_{1}\alpha}Q^{-1}_{\alpha j_{1}} \right),\\
\left(\nabla_{\bar{J}}(h_{2})_{K\bar{L}}\right)\left(\nabla^{\bar{I}}(h_{3})^{K\bar{L}} \right) =n(n+m)^{2}g_{K\bar{J}}w_{K}\overline{w}_{I},\\
\left(\nabla_{\bar{J}}(h_{3})_{K\bar{L}}\right)\left(\nabla^{\bar{I}}(h_{1})^{K\bar{L}} \right) =-m(n+m)^{2}g_{K\bar{J}}w_{K}\overline{w}_{I},\\
\left(\nabla_{\bar{J}}(h_{3})_{K\bar{L}}\right)\left(\nabla^{\bar{I}}(h_{2})^{K\bar{L}} \right) =n(n+m)^{2}g_{K\bar{J}}w_{K}\overline{w}_{I},\\
\left(\nabla_{\bar{J}}(h_{3})_{K\bar{L}}\right)\left(\nabla^{\bar{I}}(h_{3})^{K\bar{L}} \right) =(n+m)^{2}mn g_{K\bar{J}}w_{K}w_{\bar{I}}.
\end{eqnarray*}
\end{lemma}

We now consider the pairing of each of these terms with the endomorphisms $(H_{i})$.

\begin{lemma} \label{Lem:Pairing}
Let $\gamma$ be as in Equation (\ref{Eqn:Gamma_choice}) and, for ${a\in \{1,2,3\}}$, let $H_{a}$ be the endomorphisms associated to the $h_{a}$. Then
\begin{small}
\begin{eqnarray*}
 \delta_{i_{1}j_{1}}\left(U^{-1}_{j_{2}i_{2}}-U^{1}_{j_{2}\beta}U^{-1}_{\beta i_{2}}\right)(H_{1})^{\bar{J}}_{\bar{I}}  = -f_{\gamma}\left(\tr(U^{-1})-\tr(U^{-2})\right),\\
\delta_{i_{1}j_{1}}\left(U^{-1}_{j_{2}i_{2}}-U^{1}_{j_{2}\beta}U^{-1}_{\beta i_{2}}\right)(H_{2})^{\bar{J}}_{\bar{I}}  = n\left(m\left(\tr(U^{-1})-\tr(U^{-2})\right)-(n+m)\left(\tr(U^{-2})-\tr(U^{-3})\right)\right),\\
  \delta_{i_{1}j_{1}}\left(U^{-1}_{j_{2}i_{2}}-U^{1}_{j_{2}\beta}U^{-1}_{\beta i_{2}}\right)(H_{3})^{\bar{J}}_{\bar{I}} = nf_{\gamma}\left(\tr(U^{-1})-\tr(U^{-2})\right),\\ 
\delta_{i_{2}j_{2}}\left(Q^{-1}_{i_{1}j_{1}}-Q^{-1}_{i_{1}\alpha}Q^{-1}_{\alpha j_{1}} \right)(H_{1})^{\bar{J}}_{\bar{I}}  = f_{\gamma}\left(\tr(Q^{-1})-\tr(Q^{-2})\right),\\
\delta_{i_{2}j_{2}}\left(Q^{-1}_{i_{1}j_{1}}-Q^{-1}_{i_{1}\alpha}Q^{-1}_{\alpha j_{1}} \right)(H_{2})^{\bar{J}}_{\bar{I}}  = -m\left(n\left(\tr(Q^{-1})-\tr(Q^{-2})\right)-(n+m)\left(\tr(Q^{-2})-\tr(Q^{-3})\right)\right),\\
\delta_{i_{2}j_{2}}\left(Q^{-1}_{i_{1}j_{1}}-Q^{-1}_{i_{1}\alpha}Q^{-1}_{\alpha j_{1}} \right)(H_{3})^{\bar{J}}_{\bar{I}}  = mf_{\gamma}\left(\tr(Q^{-1})-\tr(Q^{-2}) \right),\\
g_{K\bar{J}}w_{K}\overline{w}_{I}(H_{1})^{\bar{J}}_{\bar{I}} =-n\tr(Q^{-1})+n\tr(Q^{-2})+(n+m)\tr(Q^{-2})-(n+m)\tr(Q^{-3}),\\
g_{K\bar{J}}w_{K}\overline{w}_{I}(H_{2})^{\bar{J}}_{\bar{I}} = m\tr(U^{-1})-m\tr(U^{-2})-(n+m)\tr(U^{-2})+(n+m)\tr(U^{-3}),\\
g_{K\bar{J}}w_{K}\overline{w}_{I}(H_{3})^{\bar{J}}_{\bar{I}} = f_{\gamma}\left(\tr(U^{-1})-\tr(U^{-2})\right).
\end{eqnarray*} 
\end{small}
\end{lemma}

Let 
\begin{equation}\label{Eqn:KoisoD1}
D_{abc} : = \langle (\nabla_{\bar{J}}h_{a})_{K\bar{L}},(H_{b})^{\bar{I}}_{\bar{J}}(\nabla_{\bar{I}}h_{c})_{K\bar{L}}\rangle,
\end{equation}
where $a,b,c, \in \{1,2,3\}$. Then, putting together the results of Lemma \ref{Lem:IBPLem1} and Lemma \ref{Lem:Pairing}, we have the expressions we need.
\begin{lemma} \label{Lem:D_ptwse}
Let $\gamma$ be as in Equation \ref{Eqn:Gamma_choice} and let $D_{abc}$ be defined by Equation (\ref{Eqn:KoisoD1}). Then
\begin{eqnarray*}
D_{111} = -m(n+m)^{2}f_{\gamma}\left(\tr(U^{-1})-\tr(U^{-2})\right),\\
D_{121} = nm(n+m)^{2}\left(m\tr(U^{-1})-m\tr(U^{-2})-(n+m)\tr(U^{-2})+(n+m)\tr(U^{-3})\right),\\
D_{131} = mn(n+m)^{2}f_{\gamma}\left(\tr(U^{-1})-\tr(U^{-2})\right),\\ 
\\
D_{112} = -(n+m)^{2}\left(-n\tr(Q^{-1})+n\tr(Q^{-2})+(n+m)\tr(Q^{-2})-(n+m)\tr(Q^{-3}) \right),\\
D_{122} = -(n+m)^{2}\left(m\tr(U^{-1})-m\tr(U^{-2})-(n+m)\tr(U^{-2})+(n+m)\tr(U^{-3}) \right),\\
D_{132} = -(n+m)^{2}f_{\gamma}\left(\tr(U^{-1})-\tr(U^{-2}) \right)\\
\\
D_{1i3} = mD_{1i2} \quad \mathrm{for} \quad i\in\{1,2,3\},\\
\\
D_{2i1} = D_{1i2} \quad \mathrm{for} \quad i\in\{1,2,3\},\\
\\
D_{212} = -nm(n+m)^{2}\left(n\tr(Q^{-1})-n\tr(Q^{-2})-(n+m)\tr(Q^{-2})+(n+m)\tr(Q^{-3})\right),\\
D_{222} = n(n+m)^{2}f_{\gamma}\left(\tr(Q^{-1})-\tr(Q^{-2})\right),\\
D_{232} = mn(n+m)^{2}f_{\gamma}\left(\tr(Q^{-1})-\tr(Q^{-2})\right),\\ 
\\
D_{2i3} = -nD_{1i2}  \quad \mathrm{for} \quad i\in\{1,2,3\},\\
\\
D_{3i1} = D_{1i3} \quad \mathrm{for} \quad i\in\{1,2,3\},\\
\\
D_{3i2} = D_{2i3} \quad \mathrm{for} \quad i\in\{1,2,3\},\\
\\
D_{3i3} = -mnD_{1i2} \quad \mathrm{for} \quad i\in\{1,2,3\}.\\
\end{eqnarray*}
\end{lemma}
Finally in this subsection we again can use Lemma \ref{Lem:integrals} to compute the $L^{2}$-inner product. Note that there is a factor of $4$ in switching from complex to Riemannian coordinates (Equation \ref{RiemCom_2}).
\begin{lemma}
Let $\mathcal{D}_{\gamma}$ be the EID as described in Theorem \ref{Thm:deform}. Then, in Riemannian coordinates, \\
$\langle \nabla_{i}\nabla_{j}(\mathcal{D}_{\gamma})_{kl},(\mathcal{D}_{\gamma})_{ij}(\mathcal{D}_{\gamma})_{kl}\rangle = $
\begin{small}
\[
 \left[\dfrac{2mn(m^{2}-1)(n^{2}-1)(m+n)^{2}(m^4n^2 + m^2n^4 - 2m^4 - 2n^4 - 4m^2n^2 + 5m^2 + 5n^2 - 4)}{n-m} \right]\int f_{\gamma}^{3}.
\]
\end{small}
\end{lemma}

\subsection{The $\langle \nabla_{i}\nabla_{j}h_{kl},h_{il}h_{kj}\rangle$ term}
As with the previous term, we return to calculating in complex coordinates and we integrate by parts 
\[
\langle \nabla_{I}\nabla_{\bar{J}}(\mathcal{D}_{\gamma})_{K\bar{L}},(\mathcal{D}_{\gamma})_{I\bar{L}}(\mathcal{D}_{\gamma})_{K\bar{J}}\rangle_{L^{2}}= -\langle \nabla_{\bar{J}}(\mathcal{D}_{\gamma})_{K\bar{L}},(\mathcal{D}_{\gamma})_{I\bar{L}}\nabla^{I}(\mathcal{D}_{\gamma})_{K\bar{J}}\rangle_{L^{2}}.
\]
Thus we will need to compute terms of the form ${(\nabla_{\bar{J}}(h_{r})_{K\bar{L}})(\nabla^{\bar{I}}(h_{s})^{K\bar{J}})}$, where ${r,s \in \{1,2,3\}}$.
\begin{lemma} \label{Lem:IBPLem2}
Let $\gamma$ be as in Lemma \ref{Lem:ZOterms} and let $h_{1}, h_{2}, $ and $h_{3}$ be as in Theorem \ref{Thm:deform}. Then
\begin{eqnarray*}
(\nabla_{\bar{J}}(h_{1})_{K\bar{L}})(\nabla^{\bar{I}}(h_{1})^{K\bar{J}}) =(n+m)^{2}g_{K\bar{L}}w_{K}\overline{w}_{I},\\
(\nabla_{\bar{J}}(h_{1})_{K\bar{L}})(\nabla^{\bar{I}}(h_{2})^{K\bar{J}}) =-n(n+m)^{2}\delta_{i_{2}l_{2}}\left(Q^{-1}_{i_{1}l_{1}}-Q^{-1}_{i_{1}\alpha}Q^{-1}_{\alpha l_{1}} \right),\\
(\nabla_{\bar{J}}(h_{1})_{K\bar{L}})(\nabla^{\bar{I}}(h_{3})^{K\bar{J}}) =-n(n+m)^{2}g_{K\bar{L}}w_{K}\overline{w}_{I},\\
(\nabla_{\bar{J}}(h_{2})_{K\bar{L}})(\nabla^{\bar{I}}(h_{1})^{K\bar{J}}) =-m(n+m)^{2}\delta_{i_{1}l_{1}}\left(U^{-1}_{l_{2}i_{2}}-U^{-1}_{l_{2}\alpha}U^{-1}_{\alpha i_{2}} \right),\\
(\nabla_{\bar{J}}(h_{2})_{K\bar{L}})(\nabla^{\bar{I}}(h_{2})^{K\bar{J}})  = (n+m)^{2}g_{K\bar{L}}w_{K}\overline{w}_{I},\\
(\nabla_{\bar{J}}(h_{2})_{K\bar{L}})(\nabla^{\bar{I}}(h_{3})^{K\bar{J}})  = m(n+m)^{2}g_{K\bar{L}}w_{K}\overline{w}_{I},\\
 (\nabla_{\bar{J}}(h_{3})_{K\bar{L}})(\nabla^{\bar{I}}(h_{1})^{K\bar{J}}) = -(n+m)^{2}\delta_{i_{1}l_{1}}\left(U^{-1}_{l_{2}i_{2}}-U^{-1}_{l_{2}\alpha}U^{-1}_{\alpha i_{2}} \right),\\
(\nabla_{\bar{J}}(h_{3})_{K\bar{L}})(\nabla^{\bar{I}}(h_{2})^{K\bar{J}}) = (n+m)^{2}\delta_{i_{2}l_{2}}\left(Q^{-1}_{i_{1}l_{1}}-Q^{-1}_{i_{1}\alpha}Q^{-1}_{\alpha l_{1}} \right),\\
(\nabla_{\bar{J}}(h_{3})_{K\bar{L}})(\nabla^{\bar{I}}(h_{3})^{K\bar{J}}) = (n+m)^{2}g_{K\bar{L}}w_{K}\overline{w}_{I}.
\end{eqnarray*}
\end{lemma}
We can then use Lemma \ref{Lem:Pairing} to complete the calculation. Let
\begin{equation}\label{Eqn:KoisoD2}
S_{abc} = \langle (\nabla_{\bar{J}}h_{a})_{K\bar{L}},(H_{b})^{\bar{I}}_{\bar{L}}(\nabla_{\bar{I}}h_{c})_{K\bar{J}}\rangle,
\end{equation}
where $a,b,c, \in \{1,2,3\}$.
\begin{lemma}\label{Lem:S_ptwse}
Let $\gamma$ be as in Equation \ref{Eqn:Gamma_choice} and let $S_{abc}$ be as in Equation \ref{Eqn:KoisoD2}. Then
\begin{eqnarray*}
S_{111} = (n+m)^{2}\left(-n\tr(Q^{-1})+n\tr(Q^{-2})+(n+m)\tr(Q^{-2})-(n+m)\tr(Q^{-3})\right),\\
S_{121} = (n+m)^{2}\left(m\tr(U^{-1})-m\tr(U^{-2})-(n+m)\tr(U^{-2})+(n+m)\tr(U^{-3})\right),\\
S_{131} = (n+m)^{2}f_{\gamma}\left(\tr(U^{-1})-\tr(U^{-2})\right),\\
\\
S_{112} = -mn(n+m)^{2}\left(-n\tr(Q^{-1})+n\tr(Q^{-2})+(n+m)\tr(Q^{-2})-(n+m)\tr(Q^{-3})\right),\\
S_{122} = -n(n+m)^{2}f_{\gamma}\left(\tr(U^{-1})-\tr(U^{-2})\right),\\
S_{132} = -mn(n+m)^{2}f_{\gamma}\left(\tr(U^{-1})-\tr(U^{-2})\right),\\
\\
S_{1i3} = -nS_{1i1}  \quad \mathrm{for} \quad i\in\{1,2,3\},\\
\\
S_{211} = m(n+m)^{2}f_{\gamma}\left(\tr(U^{-1})-\tr(U^{-2})\right),\\
S_{221} = -mn(n+m)^{2}\left(m\tr(U^{-1})-m\tr(U^{-2})-(n+m)\tr(U^{-2})+(n+m)\tr(U^{-3})\right),\\
S_{231} = -mn(n+m)^{2}f_{\gamma}\left(\tr(U^{-1})-\tr(U^{-2})\right),\\
\\
S_{2i2} = S_{1i1} \quad \mathrm{for} \quad i\in\{1,2,3\},\\
\\
S_{2i3} = mS_{2i2} \quad \mathrm{for} \quad i\in\{1,2,3\},\\
\\
S_{311} = (n+m)^{2}f_{\gamma}\left(\tr(U^{-1})-\tr(U^{-2})\right),\\
S_{321} = -n(n+m)^{2}\left(m\tr(U^{-1})-m\tr(U^{-2})-(n+m)\tr(U^{-2})+(n+m)\tr(U^{-3})\right),\\
S_{331} = ,-n(n+m)^{2}f_{\gamma}\left(\tr(U^{-1})-\tr(U^{-2})\right),\\
\\
S_{312} = m(n+m)^{2}\left(-n\tr(Q^{-1})+n\tr(Q^{-2})+(n+m)\tr(Q^{-2})-(n+m)\tr(Q^{-3})\right),\\
S_{322} = (n+m)^{2}f_{\gamma}\left(\tr(U^{-1})-\tr(U^{-2})\right),\\
S_{332} = m(n+m)^{2}f_{\gamma}\left(\tr(U^{-1})-\tr(U^{-2})\right),\\
\\
S_{3i3} = S_{1i1} \quad \mathrm{for} \quad i\in\{1,2,3\}.
\end{eqnarray*}
\end{lemma}
Again, we use Lemma \ref{Lem:integrals} to compute the $L^{2}$-inner product.Note that there is a factor of $2$ in switching from complex to Riemannian coordinates (Equation \ref{RiemCom_3}).
\begin{lemma}
Let $\mathcal{D}_{\gamma}$ be the EID as described in Theorem \ref{Thm:deform}. Then, in Riemannian coordinates, \\
$\langle \nabla_{i}\nabla_{j}(\mathcal{D}_{\gamma})_{kl},(\mathcal{D}_{\gamma})_{il}(\mathcal{D}_{\gamma})_{kj}\rangle = $
\begin{small}
\[
- \left[\dfrac{(m^{2}-1)(n^{2}-1)(m+n)^{2}(2m^{4}n^{4}-4m^{4}n^{2}-4m^{2}n^{4}+8m^{2}n^{2}+m^{4}+n^{4}-2m^{2}-2n^{2})}{n-m} \right]\int f_{\gamma}^{3}.
\]
\end{small}
\end{lemma}
\subsection{Proof of Theorem \ref{Thm:Main}}
\begin{proof}
For the special choice of $\gamma$ in Equation (\ref{Eqn:Gamma_choice}), let $\mathcal{D}_{\gamma}$ be the EID constructed in Theorem \ref{Thm:deform}. Koiso's obstruction integral \ref{Koiobs} is\\
$\mathcal{I}(\mathcal{D}_{\gamma}) =$
\begin{small}
\[
 2(n+m)\langle (\mathcal{D}_{\gamma})_{i}^{k}(\mathcal{D}_{\gamma})_{kj},(\mathcal{D}_{\gamma})_{ij}\rangle_{L^{2}}+3\langle\nabla_{i}\nabla_{j}(\mathcal{D}_{\gamma})_{kl},(\mathcal{D}_{\gamma})_{ij}(\mathcal{D}_{\gamma})_{kl} \rangle_{L^{2}}-6\langle\nabla_{i}\nabla_{l}(\mathcal{D}_{\gamma})_{kj},(\mathcal{D}_{\gamma})_{ij}(\mathcal{D}_{\gamma})_{kl}\rangle_{L^{2}}, 
\]
\end{small}
where we have used the fact that in our construction, the Einstein constant $\lambda=(n+m)$. Using Lemmas \ref{Lem:ZOterms}, \ref{Lem:IBPLem1}, and \ref{Lem:IBPLem2}, we compute
\[
\mathcal{I}(\mathcal{D}_{\gamma}) = \frac{4(m^{2}-1)^{2}(n^{2}-1)^{2}(m+n)^4(mn-1)}{n-m}\int f_{\gamma}^{3}.
\]
Thus we see that when $n\neq m$, the obstruction integral is, for {\it any} choice of ${\gamma\in \mathfrak{su}_{n+m}}$, a non-zero multiple of 
\[\int f_{\gamma}^{3}\] 
which, as discussed, does not vanish when $m\neq n$. Thus when $n+m$ is odd, the associated obstruction map in $\mathrm{Hom}_{G}(s^{2}(\mathfrak{g}),\mathfrak{g})$ does not vanish and so all the EIDs are obstructed to second order.
\end{proof}

\section{Koiso's original example revisited}\label{Sec:5}
In \cite{KoiOsaka2}, Koiso constructed EIDs for the product metric $g\oplus\tilde{g}$ on $\mathbb{CP}^{2n}\times \mathbb{CP}^{1}$. In our normalisation, given an eigenfunction $f_{\gamma}$ with eigenvalue $2(n+1)$, the tensor
\[
h=\nabla^{2}f_{\gamma}+(2n+1)f_{\gamma}g +(2n+1)(1-n)f_{\gamma}\tilde{g}. 
\]
is an EID.  To compute $\mathcal{I}(h)$, Koiso proves the following identities which we state for the Fubini--Study metric on $\mathbb{CP}^{n}$ with an arbitrary normalisation of its Einstein constant. 
\begin{lemma}[Koiso, Lemma 6.8 and Lemma 6.9 in \cite{KoiOsaka2}]
Let $g$ be the Fubini--Study metric on $\mathbb{CP}^{n}$ with Einstein constant $\lambda$ and let $f$ satisfy $\Delta f=2\lambda f$.  Furthermore, denote by $c$ the holomorphic sectional curvature of $g$ so that
\[
\mathrm{Rm}^{\;j\;l}_{i\;k} =  c(\delta_{i}^{j}\delta_{k}^{l}+\delta_{i}^{l}\delta_{k}^{j}),
\]  
then the following identities hold
\begin{eqnarray}
\langle (\nabla_{i}\nabla_{j}f)(\nabla^{j}\nabla_{k}f),(\nabla_{i}\nabla_{k}f)\rangle_{L^{2}} = \lambda^{3}\int f^{3}, \label{Eqn:HessKoi_1}\\
\langle (\nabla_{i}\nabla_{j}\nabla_{k}\nabla_{l}f),(\nabla_{i}\nabla_{j}f)(\nabla_{k}\nabla_{l}f)\rangle_{L^{2}} = -2c\lambda^{3}\int f^{3}, \label{Eqn:HessKoi_2}\\
\langle (\nabla_{i}\nabla_{k}\nabla_{j}\nabla_{l}f),(\nabla_{i}\nabla_{j}f)(\nabla_{k}\nabla_{l}f)\rangle_{L^{2}} = -c\lambda^{3}\int f^{3}. \label{Eqn:HessKoi_3}
\end{eqnarray}
\end{lemma}
We now consider these identities with respect to our construction of the Fubini--Study metric on $\mathbb{CP}^{n}$ (which has $\lambda=(n+1)$)  and for the eigenfunction $f_{\gamma}$ with $\gamma$ given by Equation (\ref{Eqn:Gamma_choice}). Using the Gasqui--Goldschmidt Lemma \ref{GG_Lem}, we have
\[
\nabla^{2}f_{\gamma} = h_{1}-h_{2}.
\]
To compute the quantity in Equation (\ref{Eqn:HessKoi_1}) we consider, recalling $Z_{abc}$ defined by Equation (\ref{EqnZ_defn}), we need to calculate
\[
\int\left( (Z_{111}-Z_{222})+(Z_{221}-Z_{112})+(Z_{212}-Z_{121})+(Z_{122}-Z_{211})\right).
\]
The results of Lemma \ref{Lem:Z_ptwse}  yield
\[
\int\left((Z_{111}-Z_{222})+(Z_{221}-Z_{112})+(Z_{212}-Z_{121})+(Z_{122}-Z_{211})\right) = \frac{(n+m)^{3}}{2}\int f_{\gamma}^{3},
\]
and so we see we recover the identity (\ref{Eqn:HessKoi_1}) on setting $m=1$ and using the complex to Riemannian coordinate scaling in Equation (\ref{RiemCom_1}). \\
\\
The identities (\ref{Eqn:HessKoi_2}) and (\ref{Eqn:HessKoi_3}) are similar; we consider $D_{abc}$ defined by (\ref{Eqn:KoisoD1}) and $S_{abc}$ defined by (\ref{Eqn:KoisoD2}), and consider the quantities
\[
-\int\left((D_{111}-D_{222})+(D_{221}-D_{112})+(D_{212}-D_{121})+(D_{122}-D_{211})\right)
\]
and 
\[
-\int\left((S_{111}-S_{222})+(S_{221}-S_{112})+(S_{212}-S_{121})+(S_{122}-S_{211})\right).
\]
The first of these, calculated with the formulae from Lemmas \ref{Lem:integrals} and \ref{Lem:D_ptwse}, yields
\[
\hspace{-25pt}-\int\left((D_{111}-D_{222})+(D_{221}-D_{112})+(D_{212}-D_{121})+(D_{122}-D_{211})\right) = 
\]
\[
\hspace{25pt}\qquad -\frac{(n+m)^{2}(mn+1)}{2}\int f_{\gamma}^{3}.
\]
The second,  calculated with the formula from Lemmas \ref{Lem:integrals} and \ref{Lem:S_ptwse}, yields
\[
\hspace{-25pt}-\int\left((S_{111}-S_{222})+(S_{221}-S_{112})+(S_{212}-S_{121})+(S_{122}-S_{211})\right)
\]
\[
\hspace{25pt}\qquad  =-\frac{(n+m)^{2}(mn+1)}{2}\int f_{\gamma}^{3}.
\]
In our normalisation, the holomorphic sectional curvature for $\mathbb{CP}^{n}$ is $c=1$ and so we recover Koiso's identities (\ref{Eqn:HessKoi_2}) and (\ref{Eqn:HessKoi_3}) after setting $m=1$ and applying the  complex to Riemannian coordinate scalings in Equations (\ref{RiemCom_2}) and (\ref{RiemCom_3}).

\bibliographystyle{acm} 
\bibliography{RigidityRefs}

\begin{thebibliography}{10}

\bibitem{AvM}
{\sc Adler, M., and van Moerbeke, P.}
\newblock Integrals over {G}rassmannians and random permutations.
\newblock {\em Adv. Math. 181}, 1 (2004), 190--249.

\bibitem{BHMW}
{\sc Batat, W., Hall, S.~J., Murphy, T., and Waldron, J.}
\newblock Rigidity of {$SU_{n}$}-{T}ype {S}ymmetric {S}paces.
\newblock {\em Int. Math. Res. Not. IMRN}, 3 (2024), 2066--2098.

\bibitem{Bes}
{\sc Besse, A.~L.}
\newblock {\em Einstein manifolds}.
\newblock Classics in Mathematics. Springer-Verlag, Berlin, 2008.
\newblock Reprint of the 1987 edition.

\bibitem{GGbook}
{\sc Gasqui, J., and Goldschmidt, H.}
\newblock {\em Radon transforms and the rigidity of the {G}rassmannians},
  vol.~156 of {\em Annals of Mathematics Studies}.
\newblock Princeton University Press, Princeton, NJ, 2004.

\bibitem{HMW}
{\sc Hall, S.~J., Murphy, T., and Waldron, J.}
\newblock Compact {H}ermitian symmetric spaces, coadjoint orbits, and the
  dynamical stability of the {R}icci flow.
\newblock {\em J. Geom. Anal. 31}, 6 (2021), 6195--6218.

\bibitem{Kan}
{\sc Kaneko, J.}
\newblock Selberg integrals and hypergeometric functions associated with {J}ack
  polynomials.
\newblock {\em SIAM J. Math. Anal. 24}, 4 (1993), 1086--1110.

\bibitem{KoiOsaka1}
{\sc Koiso, N.}
\newblock Rigidity and stability of {E}instein metrics - the case of compact
  symmetric spaces.
\newblock {\em Osaka J. Math. 17\/} (1980), 51--73.

\bibitem{KoiOsaka2}
{\sc Koiso, N.}
\newblock Rigidity and infinitesimal deformability of {E}instein metrics.
\newblock {\em Osaka J. Math. 19\/} (1982), 643--668.

\bibitem{Koi_EMCS}
{\sc Koiso, N.}
\newblock Einstein metrics and complex structures.
\newblock {\em Invent. math 73\/} (1983), 71--106.

\bibitem{KlKr1}
{\sc Kr\"oncke, K.}
\newblock Stability of {E}instein metrics under {R}icci flow.
\newblock {\em Comm. Anal. Geom. 28}, 2 (2020), 351--394.

\bibitem{Mat}
{\sc Matsushima, Y.}
\newblock Remarks on {K}\"ahler--{E}instein manifolds.
\newblock {\em Nagoya Math. J. 46\/} (1972), 161--173.

\bibitem{NagSem23}
{\sc Nagy, P.-A., and Semmelmann, U.}
\newblock Second order {E}instein deformations.
\newblock {\em preprint\/} (2023).

\bibitem{SchSem24}
{\sc Schwahn, P., and Semmelmann, U.}
\newblock On the rigidity of the complex {G}rassmannians.
\newblock {\em preprint\/} (2024).

\end{thebibliography}

\end{document}